\documentclass[12pt]{amsart}
\usepackage{amscd,amsmath,amssymb,amsfonts,verbatim}
\usepackage[cmtip, all]{xy}
\usepackage{graphicx}
\usepackage[active]{srcltx}

\newtheorem{thm}{Theorem}[section]
\newtheorem{prop}[thm]{Proposition}
\newtheorem{lem}[thm]{Lemma}
\newtheorem{cor}[thm]{Corollary}

\theoremstyle{definition}

\newtheorem{defn}[thm]{Definition}
\newtheorem{remk}[thm]{Remark}
\newtheorem{remks}[thm]{Remarks}

\newtheorem{exm}[thm]{Example}
\newtheorem{exms}[thm]{Examples}
\newtheorem{notat}[thm]{Notation}

\numberwithin{equation}{section}

{\hfill$\square$\end{defn}}

{\hfill$\square$\end{remk}}

{\hfill$\square$\end{remks}}

{\hfill$\square$\end{exm}}

{\hfill$\square$\end{exms}}

{\hfill$\square$\end{notat}}

\newcommand{\sB}{{\mathcal B}}
\newcommand{\sC}{{\mathcal C}}
\newcommand{\sD}{{\mathcal D}}
\newcommand{\sE}{{\mathcal E}}
\newcommand{\sF}{{\mathcal F}}
\newcommand{\sG}{{\mathcal G}}

\newcommand{\sK}{{\mathcal K}}

\newcommand{\sM}{{\mathcal M}}

\newcommand{\sO}{{\mathcal O}}
\newcommand{\sP}{{\mathcal P}}

\newcommand{\sS}{{\mathcal S}}

\newcommand{\sU}{{\mathcal U}}
\newcommand{\sV}{{\mathcal V}}
\newcommand{\sW}{{\mathcal W}}
\newcommand{\sX}{{\mathcal X}}
\newcommand{\sY}{{\mathcal Y}}
\newcommand{\sZ}{{\mathcal Z}}
\newcommand{\A}{{\mathbb A}}

\newcommand{\C}{{\mathbb C}}

\renewcommand{\P}{{\mathbb P}}

\newcommand{\R}{{\mathbb R}}

\newcommand{\Z}{{\mathbb Z}}

\newcommand{\fm}{{\mathfrak m}}

\newcommand{\surj}{\twoheadrightarrow}
\newcommand{\inj}{\hookrightarrow}

\newcommand{\Pic}{{\rm Pic}}

\newcommand{\Hom}{{\rm Hom}}

\newcommand{\Spec}{{\rm Spec \,}}

\newcommand{\0}{\emptyset}
\newcommand{\sHom}{{\mathcal{H}{om}}}

\newcommand{\Sch}{{\operatorname{\mathbf{Sch}}}}

\newcommand{\Top}{{\mathbf{Top}}}

\newcommand{\<}{\langle}
\renewcommand{\>}{\rangle}
\newcommand{\Sets}{{\mathbf{Sets}}}

\newcommand{\Sm}{{\mathbf{Sm}}}

\newcommand{\Ho}{{\mathbf{Ho}}}

\newcommand{\ds}{{/\kern-3pt/}}

\newcommand{\sm}{{\operatorname{sm}}}

\newcommand{\colim}{\mathop{\text{colim}}}

\newcommand{\ov}{\overline}
\newcommand{\wt}{\widetilde}
\newcommand{\wh}{\widehat}

\renewcommand{\dim}{\text{\rm dim}}

\newcommand{\tuborg}{\left\{\begin{array}{ll}}
\newcommand{\sluttuborg}{\end{array}\right.}

\newcommand{\Shv}{{\mathbf{Shv}}}
\newcommand{\PShv}{{\mathbf{PSh}}}
\newcommand{\lAff}{{\operatorname{lAff}}}

\begin{document}
\title{Motivic homotopy of group scheme actions}
\author{Amalendu Krishna, Paul Arne {\O}stv{\ae}r}



\begin{abstract}
To smooth schemes equipped with a smooth affine group scheme action,
we associate an equivariant motivic homotopy category. 
Underlying our construction is the choice of an `equivariant Nisnevich topology' induced by a complete, 
regular, 
and bounded $cd$-structure. 
We show equivariant $K$-theory of smooth schemes is represented in the equivariant motivic homotopy category. 
This is used to characterize equivariantly contractible smooth affine curves and equivariant vector bundles on such curves. 
Generalizations of the purity and blow-up theorems in motivic homotopy theory are shown for actions of finite cyclic groups.    
\end{abstract}

\maketitle
\tableofcontents

\section{Introduction}\label{section:Intro}
In this paper we develop motivic homotopy theory of smooth affine group scheme 
actions.  
We show the main results in the pioneering work of Morel-Voevodsky 
on motivic homotopy theory \cite{MV} generalize to the equivariant setting, 
e.g., the purity theorem for Thom spaces, the blow-up theorem and 
representability of $K$-theory.   

A major motivation for motivic homotopy theory of group actions is to 
construct a convenient setting for equivariant cohomology theories on 
the category of smooth schemes with a group action.
For the group of order two, 
the most important examples are Real algebraic $K$-theory, 
Real motivic cobordism \cite{HKO},  
and a Bredon type theory of equivariant motivic cohomology 
\cite{HVO}.
We complete these results by representing equivariant $K$-theory of 
group actions, as introduced by Thomason in the mid 1980's \cite{Thomason3}.
By considering actions by the multiplicative group scheme $\mathbb{G}_m$ and 
its subgroup schemes $\mu_n$ of sheaves of roots of unity, 
our results point towards an algebro-geometric version of  
$S^{1}$-equivariant homotopy theory \cite{Greenlees}.

As for the trivial group, every equivariant $\A^{1}$-homotopy equivalence becomes an 
isomorphism in the equivariant motivic homotopy category.
This basic observation implies the same result holds for every 
equivariant vector bundle, by patching of equivariant 
$\A^{1}$-homotopies on a Zariski open covering. 
Recall that for an algebraic group $G$ over a field $k$, 
a $G$-equivariant vector bundle $\sV$ over a $k$-scheme $X$ 
with $G$-action is trivial 
if there exists a $G$-representation $V$ such that $\sV=V\times_k X$. 
For $G$ a finite cyclic group, 
representability of equivariant $K$-theory allows us to show that every 
$G$-equivariant vector bundle on an equivariantly $\A^{1}$-contractible 
smooth affine curve is trivial. 
It is an open question (even when $G$ is trivial) whether the same result holds in 
higher dimensions, starting with surfaces.

We relate the equivariant motivic homotopy category $\Ho^G_{\A^{1}}(k)$ to 
other existing settings for homotopy theory.
For example, 
there is a naturally induced adjunction between $\Ho^G_{\A^{1}}(k)$ and the 
motivic homotopy category $\Ho_{\A^{1}}(k)$ corresponding to 
the trivial group.
When the base scheme is the complex numbers, 
taking complex points furnishes a `realization' functor to the equivariant 
homotopy category of topological spaces equipped with an action by 
the complex points of $G$.
Under base change the corresponding equivariant motivic homotopy categories 
are related by standard adjunctions.
These and many other functorial properties deserve a thorough treatment using 
cross-functors in the sense of Voevodsky, 
following Ayoub's work \cite{Ayoub}, 
which is somehow beyond the scope of this paper.

Our main results can be summarized as follows.
We leave precise statements to the main body of the paper. 

\begin{thm}\label{thm:Main-Results}
Let $G$ be a smooth affine group scheme over a field $k$.
Denote by $\Sm^G_k$ the category of separated finite type smooth $k$-schemes equipped with a $G$-action.
Then we have the following.
\begin{enumerate}
\item
The equivariant Nisnevich topology on $\Sm^G_k$ is given by a complete, 
regular and bounded $cd$-structure.
\item
The category of simplicial presheaves on $\Sm^G_k$ admits local and
motivic model structures in which the local weak equivalence is determined by 
the equivariant Nisnevich topology and the motivic weak equivalence is 
governed by equivariant vector bundle projections. 
\item
The equivariant algebraic $K$-theory of smooth $G$-schemes is represented
in the equivariant motivic homotopy category. Motivic
weak equivalences between smooth $G$-schemes induce isomorphisms on 
equivariant $K$-theory.
\item
If $k$ is infinite and $G$ is a finite cyclic group, 
all equivariant vector bundles on an equivariantly contractible smooth
affine curve are trivial.
\item
If $k$ is algebraically closed and $G$ is a finite cyclic group of prime order,
the purity and blow-up theorems for closed immersions of smooth $G$-schemes
hold in the equivariant motivic homotopy category.
\end{enumerate} 
\end{thm}

Herrmann \cite[Proposition~3.5.4]{Herman} proved that equivariant $K$-theory 
cannot be represented in the equivariant motivic homotopy 
category for the intermediate Nisnevich topology 
(see \S~\ref{subsubsection:H-Nis*}), 
which follows a construction common in topology by defining weak equivalences 
via fixed point loci of all subgroup schemes.
This approach to equivariant motivic homotopy theory is intuitively very clear and works well in many aspects, 
but alas does not mesh well with cohomology theories such as equivariant $K$-theory.
Our representability theorem is in some sense made possible by the fine 
differences between equivariant Nisnevich topologies, 
cf.~\S~\ref{section:CD-structure}.

\vskip .3cm

{\sl Related works:}
The first version of this paper was written in 2011 and 
several related papers have appeared during its hiatus period.
The subject of motivic homotopy of group actions can be traced back to 
Deligne's lecture notes 
\cite{Deligne} emphasizing the role of quotients by finite group actions in 
the Rost-Voevodsky proof of the Bloch-Kato conjecture.
Hu-Kriz-Ormsby \cite{HKO} used the equivariant
Nisnevich topology to introduce Real algebraic $K$-theory and Real 
motivic cobordism for the group of order two.
Herrmann \cite{Herman} worked out an unstable and stable equivariant motivic 
homotopy theory based on fixed points using the 
intermediate Nisnevich topology, cf.~\S~\ref{subsubsection:H-Nis*}.
An alternate construction carried out by Carlsson-Joshua \cite{CJ} allowing 
for actions of discrete groups is bootstrapped for solving 
Carlsson's conjecture relating $K$-theory of fields to representation theory.
A Bredon style motivic cohomology theory related to Real algebraic $K$-theory 
and equivariant higher Chow groups was introduced by 
Heller-Voineagu-{\O}stv{\ae}r \cite{HVO}.
\vspace{0.1in}

\vskip .4cm

{\sl Brief Outline of the paper:}
We describe the equivariant
Nisnevich site on $\Sch^G_k$ in \S~\ref{section:Nis-G-site} and give a
comparison of this site with other known equivariant topologies.
A $cd$-structure on the equivariant Nisnevich site and several of its 
consequences are described in \S~\ref{section:CD-structure}.

In \S~\ref{section:Models}, \ref{section:UHC} and \ref{section:UHC-P}, 
we work out the model structures on motivic $G$-spaces based 
on the equivariant Nisnevich topology and Bousfield localization with respect 
to the affine line $\A^{1}$.
A comparison with the motivic homotopy category and certain base change 
properties are investigated in \S~\ref{section:Site-Com}.
Proofs of the equivariant purity and blow-up theorems occupy 
\S~\ref{section:ENL-G-sch} and ~\ref{section:EPT}.
Finally, we prove representability of equivariant $K$-theory, 
and derive an algebraic analogue of Segal's theorem in \S~\ref{section:Nis-Desc-rep}.

\vskip .4cm

{\sl Generalizations:}
One may observe that the equivariant Nisnevich topology can be defined over 
any noetherian base scheme $S$ using the $cd$-topology defined by the
distinguished squares ~\eqref{eqn:cd-square}.
Proposition~\ref{prop:CRB} remains valid in this set up.
The results of sections~\ref{section:Models}-\ref{section:Site-Com} and 
\ref{section:Nis-Desc-rep} also generalize mutatis mutandis.

\section{The Nisnevich site for $G$-schemes}\label{section:Nis-G-site}
Let $k$ be a field and $G$ a smooth affine group scheme over $k$.
Recall that the identity component $G^0$ of $G$ is a normal closed subgroup
of $G$ which is smooth over $k$. Moreover, the quotient $\ov{G}$ is a 
finite {\'e}tale group scheme over $k$. We shall assume throughout this text 
that $\ov{G}$ is a finite constant group scheme over $k$.

Let $\Sch_k$ denote the category of separated schemes of finite type over $k$
and let $\Sm_k$ denote its full subcategory consisting of smooth
schemes over $k$.  A {\sl scheme} in this paper will mean an object of
$\Sch_k$ and the scheme $X \times Y$ will mean the fiber product of 
schemes $X$ and $Y$ over $k$.
 
Let $\Sch^G_k$ ($\Sm^G_k$) denote the category of (smooth) 
separated schemes of finite type over $k$ which are equipped with a 
$G$-action such that maps between two $G$-schemes are $G$-equivariant and 
commute with the structure maps to $\Spec(k)$. In particular, an object of 
$\Sch^G_k$ is a pair $(X, \mu_X)$ such that $X \in \Sch_k$ and there is an 
action map $\mu_X: G \times X \to X$ which satisfies the usual axioms of 
group actions.
A morphism $f : (X, \mu_X) \to (Y, \mu_Y)$ is a morphism
$f : X \to Y$ in $\Sch_k$ such that $f \circ \mu_X = \mu_Y \circ 
({\rm id}_G \times f)$. A scheme $X$ can be viewed as a $G$-scheme
via the trivial action, in which case we write it as the pair $(X, t_X)$.
In this case, $t_X$ is nothing but the projection map $t_X : G \times X \to X$.
This yields a full and faithful embedding 
\begin{equation}\label{eqn:trivial-embedding}
\iota_k : \Sch_k \to \Sch^G_k; \ \ X \mapsto (X, t_X)
\end{equation}
of categories.

\subsection{Stabilizer subgroups for $G$-actions}
\label{subsection:Stab}
Let $\Top$ denote the category of topological spaces and for a given  
topological group $G$, let ${\Top}^G$ denote the category of 
topological spaces with continuous $G$-actions. There are forgetful functors
$|-|: \Sch_k \to \Top$ and $|-| : \Sch^G_k \to {\Top}^{|G|}$.
These functors take a scheme $X$ to the underlying Zariski topological space of $X$. 
Cartesian products in $\Top$ are given the product topology.

Let $G$ be a smooth affine group scheme over $k$ and let $X \in \Sch^G_k$.
Given a point $x \in X$, the {\sl scheme-theoretic} stabilizer of $x$ is the
$k(x)$-group scheme $G_x$ defined by the Cartesian square:
\begin{equation}\label{eqn:Stab-0}
\xymatrix@C1pc{
G_x \ar[r] \ar[d] & G \times X \ar[d]^{(\mu_X, {\rm id}_X)} \\
\Spec(k(x)) \ar[r]_>>>{\Delta_X} & X \times X.}
\end{equation}

The {\sl set-theoretic} stabilizer of $x$ is the topological group
$S_x$ defined by the Cartesian square:
\begin{equation}\label{eqn:Stab-1}
\xymatrix@C1pc{
S_x \ar[r] \ar[d] & |G| \times |X| \ar[d]^{(\mu_{|X|}, {\rm id}_{|X|})} \\
|\Spec(k(x))| \ar[r]_>>>{\Delta_{|X|}} & |X| \times |X|.}
\end{equation}

There is a commutative diagram in $\Top$:
\begin{equation}\label{eqn:Stab-2}
\xymatrix@C1pc{
|G_x| \ar[r] \ar[d] & |G \times X| \ar[d]^{|(\mu_X, {\rm id}_X)|} 
\ar[r] & |G| \times |X| \ar[d]^{(\mu_{|X|}, {\rm id}_{|X|})} \\
|\Spec(k(x))| \ar[r]_>>>{\Delta_X} & |X \times X| \ar[r] & 
|X| \times |X|.}
\end{equation}
It is well known that the left square is Cartesian and the horizontal
arrows in the right square are surjective. Furthermore, these horizontal arrows 
are isomorphisms on the sets of closed points if $k$ is algebraically closed.
We conclude:

\begin{lem}\label{lem:Stab-main}
Given a point $x \in X$, there is a natural morphism of topological
groups $|G_x| \to S_x \to |G|$ such that the second morphism is an
inclusion. It is an inclusion of a closed subgroup if $x$ is a closed point of $X$. 

If $G$ is a finite constant group scheme over $k$, we can identify $G$
with $|G|$ and there are inclusions of closed subgroups $G_x \inj S_x \inj G$.
If $k$ is algebraically closed and $x$ is a closed point of $X$, then
the map $|G_x| \to S_x$ is bijective on the sets of closed points.
\end{lem}

If $G$ is a finite constant group scheme over $k$ one can get the
following explicit descriptions of $G_x$ and $S_x$. One can  
check that a $G$-action on $X$ is the same as a homomorphism of groups 
$\sigma: G \to {\rm Aut}_{\Sch_k}(X)$ and
one has $S_x = \{g \in G| g \cdot x : = \sigma(g)(x) = x\}$. 
If $g \in S_x$, then it acts on $\Spec(k(x))$ which is
just the restriction of $S_x$-action on the scheme $X$.
In other words, $\sigma$ restricts to a homomorphism $\sigma_x: S_x \to
{\rm Aut}_{\Sch}(\Spec(k(x)))$. Here, $\Sch$ denotes the category of all
schemes over $k$. One checks using ~\eqref{eqn:Stab-0} that
$G_x = {\rm Ker}(\sigma_x)$.

\begin{exm}\label{exm:Set-scheme-1}
Let $G$ be the cyclic group of order two acting on $X = \Spec(\C)$ in the category $\Sch_{\R}$ by complex conjugation.
One checks easily that $G_x$ is trivial while $S_x = G$, 
where $x$ is the unique point of $X$. 
\end{exm}

\begin{exm}\label{exm:Set-scheme-2}
If $X \in \Sch^G_k$ and $g \in S_x$ for some $x \in X$, 
observe that $g$ acts on the local ring $\sO_{X,x}$ by $k$-algebra automorphisms. 
We show by an example that $G_x$ may act as the identity on $k(x)$ and differently on $\sO_{X,x}$.
Let $G = \<\sigma \>$ be the cyclic group of order two acting on $\A^1_k$ by $\sigma(x)=-x$. 
It is clear that $S_0=G$ and hence $G_0=G$. 
This action is algebraically described by the $k$-algebra automorphism $\sigma: k[x] \to k[x]$ given by $\sigma(x) = -x$.
It is then clear that $G$ acts on the local ring $\sO_{\A^1_k, 0} \simeq k[x]_{(x)}$ by $\sigma(x) = -x$. 
This is the identity map if and only if ${\rm char}(k) = 2$.
\end{exm}

\vskip .3cm

\subsection{Equivariant Nisnevich topology}\label{subsection:ENM}
We now define our Nisnevich site for $G$-schemes in terms of 
equivariant Nisnevich coverings. We shall show that these coverings
yield a Grothendieck topology on $\Sch^G_k$.
The exact same definitions and results hold for $\Sm^G_k$.

\begin{defn}\label{defn:Nisne-cover}
Let $X \in \Sch^G_k$. A family of \'etale morphisms 
$\{Y_i \xrightarrow{f_i} X\}_{i \in I}$ in 
$\Sch^G_k$ is called a ($G$-equivariant) Nisnevich cover of $X$ if
for any point $x \in X$, there is an index $i = i(x) \in I$ and a point
$y \in Y_i$ such that 
\begin{enumerate}
\item
$f_i(y) = x$,
\item
the induced map of the residue fields $k_x \to k_y$ is an isomorphism,
and
\item
the induced map $S_y \to S_x$ is an isomorphism.
\end{enumerate}
\end{defn}

It is immediate from this definition that a $G$-equivariant 
Nisnevich cover is the same as a Nisnevich cover in the sense of
\cite{MV} if $G$ is trivial.

\begin{prop}\label{prop:Nisne-site}
The category $\Sch^G_k$ with the $G$-equivariant Nisnevich coverings 
constitutes a Grothendieck site.
\end{prop}

\begin{proof}
By the definition it is clear that $G$-equivariant isomorphisms are Nisnevich coverings, 
and any refinement of a $G$-equivariant Nisnevich cover is also of the same type. 
We need to check that coverings are preserved under base change. 
This part is not automatically true by reduction to ordinary (non-equivariant) Nisnevich covers because for some point $x$, 
a point $y$ in the cover mapping to $x$ may satisfy condition (2) of Definition~\ref{defn:Nisne-cover}, 
but not condition (3).
\vspace{0.1in}

We consider the Cartesian diagram in $\Sch^G_S$
\begin{equation}\label{eqn:Base-change}
\xymatrix@C.9pc{
W \ar[r]^{v'} \ar[d]_{u'} & Y \ar[d]^{u} \\
Z \ar[r]_{v} & X,}
\end{equation}
where $u$ is a $G$-equivariant Nisnevich cover. 
It is clear that $u'$ is {\'e}tale.
Let us now fix a point $z \in Z$ and let $v(z) = x$. Choose a point
$y \in Y$ such that $k(x) \xrightarrow{\simeq} k(y)$ and $S_y 
\xrightarrow{\simeq} S_x$.

It is easy to check (see \cite[Exercise~3.1.7]{Liu}) that there is a natural
homeomorphism of topological spaces
\[
\Spec (k(y) {\underset{k(x)}\otimes} k(z))
\xrightarrow{\simeq} \{w \in W| v'(w) = y, \ u'(w) = z\}.
\]

On the other hand, the isomorphism $k(x) \xrightarrow{\simeq} k(y)$ implies 
that the map $k(z) \to k(z) {\underset{k(x)}\otimes} \ k(y)$ is an
isomorphism. In particular, $k(z) {\underset{k(x)}\otimes} \ k(y)$ is a
field and defines a unique point $w = (y,z) \in W$ such that
$v'(w) = y$ and $u'(w) = z$. Moreover, the map
$k(z) \to k(w)$ is an isomorphism.
We are only left with showing that $S_w \rightarrow S_z$ is an isomoprhism.
This map is clearly injective. To prove its surjectivity,
notice that $g \in S_z$ implies $g \in S_x$ as $v(z) = x$ and $v$ is
$G$-equivariant. But then, our assumption implies
$g \in S_y$. Since $G$ acts diagonally on $W$ we conclude that
$g \in S_w = S_{(y, z)}$. 
\end{proof}

\vskip .3cm

{\bf Notations:} For the rest of this text, we shall abbreviate the term
`equivariant Nisnevich topology' by simply calling it the
{\sl $eN$-topology}. An equivariant Nisnevich cover of a $G$-scheme will
be called an $eN$-cover.
We shall denote the ($G$-equivariant) Nisnevich Grothendieck site 
on the category of $G$-schemes over $k$ by $\Sch^G_{k/{\rm Nis}}$,
and the corresponding site of smooth $G$-schemes by
$\Sm^G_{k/{\rm Nis}}$. We refer to these sites as {\sl $eN$-sites}. 
Throughout the text the following notations will be used.
\begin{enumerate}
\item
$\PShv^G_{\Sch_k} :=$ the category of presheaves of sets on $\Sch^G_k$.
\item
$\Shv^G_{\Sch_k} :=$ the category of sheaves of sets on $\Sch^G_{k/{\rm Nis}}$.
\item
$\PShv^G_{\Sm_k} :=$ the category of presheaves of sets on $\Sm^G_k$.
\item
$\Shv^G_{\Sm_k} :=$ the category of sheaves of sets on $\Sm^G_{k/{\rm Nis}}$.
\end{enumerate}

Suppose $\sC$ and $\sD$ are Grothendieck sites.
A functor $f^{-1} : \sC \to \sD$ is a {\sl continuous map of sites} if for every sheaf $F$ on $\sD$,
the presheaf $f_*(F) = F \circ f^{-1}$ is a sheaf on $\sC$.
Such a map of sites is written $f: \sD \to \sC$.
A continuous map of sites $f$ is called a {\sl morphism of sites} if the left adjoint $f^*$ of $f_*$ commutes with finite limits.
Since we shall discuss functors between Grothendieck sites, 
the following criterion for these notions will be used repeatedly in order to decide about the nature of these functors.

\begin{prop}$($\cite[Remarks~1.1.44, 1.1.45]{MV}$)$\label{prop:site-N}
Suppose the functor $f^{-1}: \sC \to \sD$ commutes with fiber products.
\begin{enumerate}
\item
Then the map of sites $f: \sD \to \sC$ is continuous if and only if $f^{-1}$ preserves coverings. 
\item
Suppose furthermore the topology on $\sD$ is sub-canonical. 
Then $f$ is a morphism of sites if and only if it is continuous.
\end{enumerate}
\end{prop}

It follows that there is a continuous map of sites $\tau_G: \Sch^G_{S/{\rm Nis}} \to \Sm^G_{S/{\rm Nis}}$. 
One knows, 
however, 
that $\tau_G$ is not a morphism of sites (see \cite[Example~I.1.46]{MV}).

\subsection{Comparison with other topologies on $G$-schemes}
\label{subsection:H-Nis}

\subsubsection{The intermediate Nisnevich topology}
\label{subsubsection:H-Nis*}
Suppose $G$ is a finite constant group scheme over $k$. 
Replacing set-theoretic stabilizers by scheme-theoretic stabilizers in 
condition (3) of Definition~\ref{defn:Nisne-cover} yields
a Grothendieck topology on $\Sm^G_k$.
This topology is called the {\sl $H$-Nisnevich topology} on $\Sm^G_k$ by 
Herrmann \cite{Herman} and the intermediate 
Nisnevich topology by Williams \cite{Williams}. 
We let $\Sch^{G-H}_{k/{\rm Nis}}$ denote the corresponding site, 
and note below that the intermediate Nisnevich topology on $\Sch^G_k$ is finer 
than the $eN$-topology.

\begin{lem}$($\cite[Lemma~2.1.14]{Herman}$)$\label{lem:H-Nis-E-Nis}
Let $f : Y \to X$ be a $G$-equivariant morphism of schemes.
Let $x \in X$ and suppose that there is a point $y \in Y$ such that
$f(y) =x$, $k(x) \xrightarrow{\simeq} k(y)$ and $S_y \xrightarrow{\simeq} S_x$.
Then there is a naturally induced isomorphism $G_y \xrightarrow{\simeq} G_x$.
\end{lem}
\begin{proof}
Let $g \in G_x$ and consider the commutative diagram
\begin{equation}\label{eqn:HE-1}
\xymatrix@C.8pc{
k(x) \ar[r]^{id} \ar[d]_{\simeq} & k(x) \ar[d]^{\simeq} \\
k(y) \ar[r]_{g_*} & k(y).}
\end{equation}
Since the vertical arrows are induced by $f$, it follows easily from this 
diagram that $g_*$ is the identity. That is, $g \in G_y$. 
\end{proof}

\subsubsection{The Isovariant Nisnevich topology}\label{subsubsection:IsoV}
Recall that for $X \in \Sch^G_k$, the isotropy group scheme is a group scheme
$G_X$ over $X$ defined by the Cartesian square 
\begin{equation}\label{eqn:Isotropy}
\xymatrix@C.9pc{
G_X \ar[r] \ar[d]_{i_X} & G \times X \ar[d]^{(\mu_X, {\rm id}_X)} \\
X \ar[r]_<<<{\Delta_X} & X \times X.}
\end{equation}

A $G$-equivariant {\'e}tale cover $\{X_i \to X\}_{i \in I}$ is called 
isovariant if the induced map of isotropy group schemes is an isomorphism
for each $i \in I$. An isovariant {\'e}tale cover which is also Nisnevich,
is called an isovariant Nisnevich cover. The isovariant {\'e}tale site
on smooth schemes was introduced by Thomason \cite{Thomason1} in order
to prove the {\'e}tale descent for Bott-inverted equivariant $K$-theory
with finite coefficients. Its Nisnevich analogue was introduced by Serpe
\cite{Serpe} in an attempt to prove descent theorems for equivariant
algebraic $K$-theory with integral coefficients. (However, most of the results 
claimed in \cite{Serpe} are either false or need amendments.)
The intermediate Nisnevich topology is clearly finer than the isovariant
Nisnevich topology. Let $\Sch^{G-Iso}_{k/{\rm Nis}}$ denote the isovariant
Nisnevich site on $\Sch^G_k$.

It is known (see \cite[Lemma~3.2.8]{Herman}) that the intermediate Nisnevich 
topology on $\Sch^G_k$ (for $G$ finite) is sub-canonical. 
It follows that the isovariant 
Nisnevich topology (being coarser than the intermediate Nisnevich topology) 
is also sub-canonical (see Corollary~\ref{cor:sub-can} for a more general 
result). 
We conclude from Proposition~\ref{prop:site-N} and 
Lemma~\ref{lem:H-Nis-E-Nis} that for $G$ finite, the identity functor on
$\Sch^G_k$ induces morphisms of Grothendieck sites
\begin{equation}\label{eqn:H-E-0}
\xymatrix@C.8pc{
\Sch^{G-H}_{k/{\rm Nis}} \ar[rr]^{\epsilon^G} \ar[dr]_{\nu^G} & & 
\Sch^G_{k/{\rm Nis}} \\
& \Sch^{G-Iso}_{k/{\rm Nis}}. &}
\end{equation}

The following examples show that the equivariant 
Nisnevich topology is distinct from the intermediate Nisnevich 
topology.  Moreover, the equivariant Nisnevich and the isovariant Nisnevich 
topologies are in general not comparable.

\begin{exm}\label{exm:Ison-H-1}
We view the complex numbers $\C$ as an $\R$-algebra and consider the map of $\R$-algebras
\[
f : \C \to \C \times \C ; \ \ a \mapsto \left(a, \bar{a}\right).
\]
Let $G=\<\sigma\>$ be the cyclic group of order two acting by complex conjugation on $\C$ and 
by switching the coordinates on $\C\times\C$. 
Note that $f$ is a $G$-equivariant $\R$-algebra map and an isovariant Nisnevich covering. 
Let $Y = \Spec(\C \times \C)$ and $X = \Spec(\C)$.
For the unique point $\eta \in X$ we have $S_{\eta} = G$. 
On the other hand, 
the set-theoretic stabilizer of any point in $f^{-1}_*(\eta)$ is trivial. 
Hence $f_*: Y \to X$ is not an $eN$-cover. 
\end{exm}

\begin{exm}\label{exm:Nis-No-Isov}
For the inclusion of $\R$-algebras $\R \to \C$ we let $G = \<\sigma\>$ (as above) act trivially on $\R$ and by 
complex conjugation on $\C$.
The inclusion is $G$-equivariant \'etale, 
but it is neither isovariant nor Nisnevich. 
However, 
the map $\R \to \R \times \C$ is a $G$-equivariant Nisnevich (hence an intermediate Nisnevich) cover of $\Spec(\R)$, 
which is not isovariant since the first map is not isovariant.
\end{exm}

The intermediate Nisnevich topology resemblances closely the situation in 
topology in the sense that 
$Y \to X$ is an intermediate Nisnevich cover if and only if the induced maps 
of fixed point loci $Y^H \to X^H$ 
are ordinary Nisnevich covers for all subgroups $H \subseteq G$. 
On the other hand, 
it is also known (see \cite[Remark~3.5.5]{Herman}) that descent and 
representability of equivariant $K$-theory 
fail in the intermediate and isovariant Nisnevich topologies.
This makes the $eN$-topology more suitable for studying cohomology theories 
for schemes with group actions.
We have also observed that coverings in the intermediate and isovariant 
Nisnevich topologies do not necessarily split.
It is unlikely that these topologies arise from $cd$-structures.

\section{A $cd$-structure on the $eN$-topology\/}
\label{section:CD-structure}
Our goal in this section is to show that the $eN$-topology
can be described in terms of a $cd$-structure in the sense of
Voevodsky \cite{Voev1}. We shall further show that this $cd$-structure is
in fact regular, complete, and bounded. Applications of this will appear
later in the paper.

\subsection{$eN$-neighborhoods}\label{subsection:Nis-N}
Let $X \in \Sch^G_k$ and let $i: Z \inj X$ be a $G$-invariant locally closed
subset with the reduced subscheme structure.  
Let us denote this datum by $(X, Z)$. 
An $eN$-neighborhood of $(X, Z)$ is a commutative
square
\begin{equation}\label{eqn:Nis-N-0}
\xymatrix@C1pc{
Z' \ar[r]^{i'} \ar[d]_{\simeq} & U \ar[d]^{f} \\
Z \ar[r]_{i} & X}
\end{equation}
in $\Sch^G_k$ such that $f$ is {\'e}tale. We shall denote such a 
neighborhood by $(U, Z)$. If the square ~\eqref{eqn:Nis-N-0} is Cartesian,
we shall say that $(U, Z)$ is a {\sl distinguished} $eN$-neighborhood of
$(X, Z)$. Notice that in this case, $Z'$ is automatically reduced.
If $U$ has no $G$-action or $f$ is not
necessarily $G$-equivariant, then we shall say that $(U, Z)$ is a Nisnevich
neighborhood of $(X, Z)$.

Given an $eN$-neighborhood $f: (U, Z) \to (X, Z)$ and a $G$-invariant
locally closed subset $Y \subseteq X$, we shall write the $G$-scheme
$Y {\underset{X}\times} \ U$ in short as $Y \cap U$ or $Y_U$.

\subsubsection{$eN$-neighborhood refinement}\label{subsubsection:ENR}
Assume that $G = \{e = g_0, \cdots , g_n\}$ is a finite constant group
scheme over $k$. 
Given a Nisnevich neighborhood $(U, Z)$ (not necessarily $G$-invariant)
and $g \in G$, the {\sl translate} of $U$ by $g$ is the scheme $g(U)$ defined
by the Cartesian square
\begin{equation}\label{eqn:Nis-N-1}
\xymatrix@C1pc{
g(U) \ar[r] \ar[d]_{g(f)} & U \ar[d]^{f} \\
X \ar[r]_{\tau_{g^{-1}}} & X,}
\end{equation}
where $\tau_{g^{-1}}: X \to X$ is the automorphism of $X$ defined by $g^{-1}$ via
the $G$-action on $X$.  
We can iteratively form the fiber product 
$$
U_G:= U {\underset{X}\times} \ g_1(U)
{\underset{X}\times} \cdots {\underset{X}\times} \ g_n(U),
$$
using the maps $g_i(f):g_i(U) \rightarrow X$. 
Since $Z \inj X$ is $G$-invariant,
it is easy to check that
$(U_G, Z)$ is in fact an $eN$-neighborhood of $(X, Z)$
and there is a factorization
$(U_Z, Z) \to (U, Z) \to (X, Z)$. We conclude that every Nisnevich
neighborhood of $(X, Z)$ contains, i.e., is dominated by, an 
$eN$-neighborhood.

\vskip .3cm

\subsection{$cd$-structure on $\Sch^G_k$}
\label{section:cd-str}
The notion of $cd$-structure on Grothendieck sites was introduced by
Voevodsky in \cite{Voev1} in order to streamline the study of homotopy 
theory of schemes with respect to various topologies. 
We refer to \cite{Voev1} for the definition of $cd$-structure on a
category and its various properties.

\begin{defn}\label{defn:dist-square}
A {\sl distinguished $eN$-square} is a commutative diagram in $\Sch^G_k$
\begin{equation}\label{eqn:cd-square}
\xymatrix@C.9pc{
B \ar[r] \ar[d] & Y \ar[d]^{p} \\
A \ar[r]_{j} & X,} 
\end{equation}
with $j$ an open immersion and $(Y, (X \setminus A)_{\rm red})$ a distinguished 
$eN$-neighborhood of $(X, (X \setminus A)_{\rm red})$.

The equivariant Nisnevich $cd$-structure on $\Sch^G_k$ is the collection of 
distinguished $eN$-squares (\ref{eqn:cd-square})
\end{defn}

It is straightforward to check that we obtain a $cd$-structure on $\Sch^G_k$ 
in the sense of \cite{Voev1}, i.e.,  
a commutative diagram isomorphic to a distinguished $eN$-square is again a 
distinguished $eN$-square.
The equivariant Nisnevich $cd$-structure on $\Sm^G_k$ is defined in the same 
way using distinguished $eN$-squares in the smooth category.
Our next result is an equivariant analogue of Voevodsky's 
\cite[Theorem~2.2]{Voev2}.
The proof is obtained by following the steps in the non-equivariant case with 
suitable modifications at various stages. 
We refer to \cite[\S~2]{Voev1} for the definition of a complete, regular, and 
bounded $cd$-structure.

\begin{prop}\label{prop:CRB}
The equivariant Nisnevich $cd$-structures on $\Sch^G_k$ and $\Sm^G_k$ 
are complete, regular, and bounded.
\end{prop}
\begin{proof}
We write a proof for the category $\Sch^G_k$ as the smooth case is no different.
The completeness is a direct consequence of \cite[Lemma~2.4]{Voev1} since the
distinguished $eN$-squares of the form
~\eqref{eqn:cd-square} are closed under pullbacks.
 
To prove regularity, we observe that 
given a distinguished $eN$-square ~\eqref{eqn:cd-square} in $\Sch^G_k$, the 
derived square
\begin{equation}\label{eqn:CB*1}
\xymatrix@C.8pc{
B \ar[r]^{e'} \ar[d]_{\Delta_B} & Y \ar[d]^{\Delta_Y} \\
B \times_A B \ar[r] & Y \times_X Y}
\end{equation}
is a distinguished square in $\Sch_k$ by \cite[Theorem~2.2]{Voev2} and hence, a 
distinguished square in $\Sch^G_k$ since all the underlying maps
in ~\eqref{eqn:CB*1} are $G$-equivariant. The regularity condition now follows
from \cite[Lemma~2.11]{Voev1}.

The boundedness condition is not straightforward from the non-equivariant case.
First we need to define a density structure on $\Sch^G_S$. 
For $X \in \Sch^G_S$ and $i \ge 0$, 
let $D_i(X)$ denote the class of open embeddings $U \to X$ in $\Sch^G_S$ that define an element of the density structure on $\Sch_S$ 
\cite[Proposition~2.10]{Voev1} under the forgetful functor $\Sch^G_S \to \Sch_S$.
That is, 
for every $z\in X\setminus U$ there exists a sequence of points $z=x_0,x_1,\dots,x_i$ in $X$ such that for $0\leq j<i$, 
$x_j\neq x_{j+1}$ and $x_j\in\overline{\{x_{j+1}\}}$. 
One verifies easily that this defines a density structure on $\Sch^G_S$, 
and it is locally of finite dimension. 

To prove boundedness, 
it is enough to show that every distinguished $eN$-square is reducing with respect to the above density structure.
Consider a distinguished $eN$-square of the form~\eqref{eqn:cd-square} and suppose $B_0 \in D_{i-1}(B), A_0 \in D_i(A)$ and $Y_0 \in D_i(Y)$.
Applying Lemma~\ref{lem:Density} below to the morphism $j \coprod p$ we can find $X_0 \in D_i(X)$ such that $j(A_0) \cap p(Y_0) \subseteq X_0$.
Replacing $Y$ by $Y_0$, $A$ by $A_0$, $B$ by $B'= A_0 {\underset{X}\times} Y_0$, $X$ by $X_0$, and applying \cite[Lemma~2.5]{Voev2} we are 
reduced to consider the distinguished $eN$-square    
\begin{equation}\label{eqn:cd-square-1}
\xymatrix@C.9pc{
B' \ar[r] \ar[d] & Y_0 \ar[d]^{p} \\
A_0 \ar[r]_{j} & X_0.} 
\end{equation}

We now set 
\begin{equation}\label{eqn:cd-square-2}
B'_0 = B' \cap B_0, \ Z = B' \setminus B'_0,  \ Y' = Y_0 \setminus 
cl_{Y_0}(Z),  \ A' = A_0 \ {\rm and} \ X' = j(A_0) \cup p(Y').
\end{equation}
In \cite[Proposition~2.10]{Voev2} it is noted that 
\begin{equation}\label{eqn:cd-square-3}
\xymatrix@C.9pc{
B'_0 \ar[r] \ar[d] & Y' \ar[d]^{p} \\
A_0 \ar[r]_{j} & X'} 
\end{equation}   
is a distinguished Nisnevich square which satisfies the required properties.
To complete the proof we observe that the inclusions in ~\eqref{eqn:cd-square-2} are $G$-invariant.
\end{proof}

\begin{lem}\label{lem:Density}
Let $f: X \to Y$ be a morphism in $\Sch^G_k$ and assume that there exists
a $G$-invariant dense open subset $U$ in $Y$ such that $f^{-1}(U)$ is dense
and $f^{-1}(U) \to U$ has fibers of dimension zero. Then for any
$i \ge 0$ and $V \in D_i(X)$, there exists $W \in D_i(Y)$ such that
$f^{-1}(W) \subseteq V$.
\end{lem}
\begin{proof}
By \cite[Lemma~2.9]{Voev2}, there exists $W' \in D_i(Y)$ such that
$f^{-1}(W') \subseteq V$. But $W'$ may not be $G$-invariant.
However, since $V \subseteq X$ is $G$-invariant (by
definition of our density structure), 
it follows that $f^{-1}(GW') = G(f^{-1}(W')) \subseteq V$.
Since the map $\mu_Y: G \times Y \to Y$ is smooth and in particular open, 
we see that $GW' \subseteq Y$ is a $G$-invariant open subset.

Setting $W = GW'$,
it is clear that $W \subseteq Y$ is a $G$-invariant open subset
such that $f^{-1}(W) \subseteq V$. Furthermore, as $W' \subseteq W$ and
$W' \in D_i(Y)$, we see that $W \in D_i(Y)$. This proves the lemma.
\end{proof}

\subsection{$cd$-property of the $eN$-topology}\label{subsection:CD-eN}
In order to show that the $eN$-topology on $\Sch^G_k$ (and $\Sm^G_k$)
is induced by the above $cd$-structure, we need to produce a splitting of 
$eN$-covers. We do this in the next result. 
Recall that $G$ is a smooth affine group scheme over $k$ such that 
$\ov{G} = G/{G^0}$ is a constant group scheme over $k$.
We write $G = \stackrel{r}{\underset{i =0}\coprod} g_iG^0$,
where $\{e = g_0, g_1, \cdots , g_r\}$ are points in $G(k)$ which
represent the left cosets of $G^0$.

\begin{defn}\label{defn:splitting}
A family of morphisms $\{Y_i\xrightarrow{f_i} X\}_{i \in I}$ in $\Sch^G_S$ \sl{splits} 
if there is a filtration of $X$ by $G$-invariant closed subschemes
\begin{equation}\label{eqn:split*-0}
\emptyset = X_{n+1} \subsetneq X_{n} \subsetneq \cdots \subsetneq X_0 = X, 
\end{equation}
and for each $0 \le j \le n$ there is an $i= i(j) \in I$ such that the map
\[
\left(X_j \setminus X_{j+1}\right) \times_X Y_i \to X_j \setminus X_{j+1}
\]
has a $G$-equivariant section. 
If each $f_i$ is also \'etale, 
the family of morphisms is called a \sl{split \'etale cover} of $X$.
\end{defn}

\begin{prop}\label{prop:Nisne-split}
A family of morphisms $\{Y_i \xrightarrow{f_i} X\}_{i \in I}$ in $\Sch^G_k$ is 
an $eN$-cover if and only if it is a split \'etale cover.
\end{prop}
\begin{proof}
It is clear that a split \'etale $G$-equivariant family of morphisms is
an $eN$-cover. The core of the proof is to show the converse.

Suppose $\{Y_i \xrightarrow{f_i} X\}_{i \in I}$ is a $G$-equivariant Nisnevich 
cover of $X$. 
Let $Z \subset X$ be the closed subscheme
(with reduced structure) which is the union of all possible nonempty 
intersections (if there are any) of the irreducible components of $X$. 
It is easy to check that
$Z$ is $G$-invariant. This follows from the fact that every 
left coset $g_iG$ takes any given irreducible component $X_j$ of $X$ onto some
(same or different) irreducible component of $X$ and
$g_iGX_j = g_iGX_{j'}$ if and only if $X_j = X_{j'}$.
Let $W$ be the 
$G$-invariant open subscheme of $X$ given by the complement of $Z$ and set
$U_i = Y_i \times_X W$. Then $\{U_i \xrightarrow{f_{U_i}} W\}$ is 
an $eN$-cover of $W$. 
Notice that $W$ is a disjoint union of its irreducible components and each
$f_{U_i}$ being \'etale, it follows that each $U_i$ is also a disjoint union of
its irreducible components.

Let $x \in W$ be a generic point of $W$. Then the closure $W_x = 
\ov{\{x\}}$ in $W$ is an irreducible component of $W$. By our assumption,
there is a point $y$ lying in some $U_i$ such that 
\begin{equation}\label{eqn:split-0}
f_{i}(y) = x, \ k_x \xrightarrow{\simeq} k_y,  \
{\rm and} \ S_y \xrightarrow{\simeq} S_x.
\end{equation}
Then the closure $U_y = \ov{\{y\}}$ in $U_i$ is an irreducible component of 
$U_i$. Since $U_y \to W_x$ is \'etale and generically an isomorphism, it
must be an open immersion. Thus $f_i$ maps $U_y$ isomorphically onto an
open subset of $W_x$. We replace $W_x$ by this open subset $f_i(U_y)$
and call it our new $W_x$. 

Let $GU_y$ be the image of the action morphism $\mu: G \times U_y \to U_i$.
Notice that $\mu$ is a smooth map and hence open.
This in particular implies that $GU_y$ is a $G$-invariant open subscheme
of $U_i$ as $U_y$ is one of the disjoint irreducible components of $U_i$
and hence open. By the same reason, $GW_x$ is a $G$-invariant open
subscheme of $W$.  

Since the identity component $G^0$ is connected, it keeps $U_y$ invariant. In other words,
the point $y \in U_i$ is fixed by $G^0$ and hence $G$ acts on this point
via its quotient $\ov{G} = G/{G^0}$.
Recall that $\ov{G}$ is a finite constant group scheme over $k$.

Since each $g_jG$ takes $U_y$ onto an irreducible component of $U_i$
and since $U_i$ has only finitely many irreducible components which are all
disjoint, we see that $GU_y = U_{i_0} \coprod U_{i_1} \coprod \cdots 
\coprod U_{i_n}$ is a disjoint union of some irreducible components of
$U_i$ with $U_{i_0} = U_y$. In particular, for each $U_{i_j}$, we have
$U_{i_j} = g_jGU_y = g_jU_y$ for some $g_jG$. 

Since $f_i$ maps $U_y$ isomorphically onto $W_x$, we conclude from the above 
that $f_i$ maps each $U_{i_j}$ isomorphically onto one and only one
$W_{j}$ such that $GW_x = f_i\left(GU_y\right) =
W_{0} \coprod W_{1} \coprod \cdots \coprod W_{m}$ (with $m \le n$) is a 
disjoint union of open subsets of some irreducible components of $W$ with 
$W_{0} = W_x$. 
The morphism $f_i$ will map the open subscheme $GU_y$ isomorphically onto the 
open subscheme $GW_x$ if and only if no two components of $GU_y$ are mapped
onto one component of $GW_x$. This is ensured by using the
third condition of the definition of the $eN$-covering.

If two distinct components of $GU_y$ are mapped onto one component of $GW_x$,
we can (using the equivariance of $f_i$) apply automorphisms by $g_j$'s 
and assume that one of these
components is $U_y$. In particular, we find that there are some
$j, j' \ge 1$ such that
\begin{equation}\label{eqn:split-1}
W_x = f_i\left(U_y\right) = f_i\left(U_{i_j}\right)
= f_i\left(g_{j'} U_{y}\right) = g_{j'} f_i\left(U_{y}\right) = g_{j'}W_x.
\end{equation}

But this implies that $g_{j'} \in S_x$ and $g_{j'} \notin S_y$. 
This violates the condition in ~\eqref{eqn:split-0} that the set-theoretic
stabilizers $S_y$ and $S_x$ are isomorphic.     
We have thus shown that the morphism $f_i$ has a $G$-equivariant splitting 
over a nonempty $G$-invariant open subset $GW_x$.
Letting $X_1$ be the complement of this open subset in $X$, we see that
$X_1$ is a proper $G$-invariant closed subscheme of $X$ and by
restricting our $eN$-cover to $X_1$, we get such a cover for
$X_1$. The proof of the proposition is now completed by the Noetherian
induction.   
\end{proof}

\begin{remk}\label{remk:Set-scheme}
One cannot conclude from ~\eqref{eqn:split-1} that
$g_j$ lies in the scheme-theoretic stabilizer of $x$. 
We thank Ben Williams for pointing this out soon after the first version of this paper 
was shared with him in 2011.
\end{remk}

\begin{remk}\label{remk:No-split}
One can easily check that Example~\ref{exm:Ison-H-1} is also  an
example of an isovariant Nisnevich cover (hence and intermediate Nisnevich 
cover) which can not admit an equivariant splitting. 
This provides a counterexample to \cite[Lemma~2.12]{Serpe}. 
\end{remk}

\begin{remk}\label{remk:Nisne-split1}
It is straightforward to see that a split \'etale cover
has the base change property. Using Proposition~\ref{prop:Nisne-split}, this 
gives  another proof of Proposition~\ref{prop:Nisne-site}.
\end{remk}

\vskip .3cm

\begin{prop}\label{prop:Comparison**}
The $eN$-topology on $\Sch^G_k$ and $\Sm^G_k$ coincides with the
topology induced by the equivariant Nisnevich $cd$-structure. 
\end{prop}
\begin{proof}
It is easy to see from the definitions that for a distinguished
square ~\eqref{defn:dist-square}, the family $\{Y \xrightarrow{p} X,
A \xrightarrow{j} X\}$ is an $eN$-cover of $X$. So we only need to
prove that any $eN$-cover has a refinement which is
an equivariant Nisnevich $cd$-cover. 
Let $\{Y_i \xrightarrow{f_i} X\}_{i \in I}$ be an $eN$-cover of $X$. 
By Proposition~\ref{prop:Nisne-split}, we can
assume that this is a split \'etale cover. In particular,
there is a finite filtration of $X$ by the $G$-invariant closed subschemes such
that the covering map is split in the complementary open subsets.
We prove our assertion by induction on the minimal length of this splitting.

If the length of the splitting is zero, then the cover has an equivariant
section $s : X \to Y_i$ for some $i \in I$. 
Since each $f_i$ is \'etale, $s$ must be {\'e}tale too. 
In particular, this section maps $X$ isomorphically onto a 
$G$-invariant open subscheme $X'$ of $Y_i$. In this case, the square
\[
\xymatrix@C.9pc{
X' \ar@{=}[r] \ar[d] & X' \ar[d] \\
X \ar@{=}[r] & X} 
\]
is a distinguished $eN$-square which refines our cover.
To conclude, it suffices now to construct a 
distinguished $eN$-square of the form ~\eqref{eqn:cd-square} 
such that the pullback of the covering map $\{Y_i \xrightarrow{f_i} X\}$ to 
$Y$ has a $G$-equivariant section and the pullback to $A$ has an equivariant 
splitting sequence of length strictly less than $n$.

Given the splitting sequence of ~\eqref{eqn:split*-0}, we see that  
$\{X_{n} \times_X Y_i \to X_n\}$ is an $eN$-cover with a $G$-equivariant section 
$s: X_n \to X_{n} \times_X Y_i$ for some $i$.
Let $X'_n$ be the image of this section. We have seen above that $X'_n$ is a
$G$-invariant open subscheme of $X_{n} \times_X Y_i$.
In particular, its complement $W_n$ is a $G$-invariant
closed subscheme. By setting $A = X \setminus X_n$ and 
$Y =  (X_{n} \times_X Y_i) \setminus W_n$, 
we see that the square defined by 
$\{A \xrightarrow{j} X, Y \xrightarrow{p} X\}$ is a distinguished $eN$-square. 
Furthermore, the pullback of this square to
$Y$ has a $G$-equivariant section and its pullback to $A$ is an $eN$-cover 
which has a splitting sequence of length less than $n$. This completes the 
proof of the proposition.
\end{proof}    

Combining Propositions~\ref{prop:CRB} and ~\ref{prop:Comparison**},
we get the following results.

\begin{thm}\label{thm:Comparison}
The $eN$-topology on $\Sch^G_k$ and $\Sm^G_k$ is induced by a 
$cd$-structure, which is complete, regular, and bounded.
\end{thm}

\begin{cor}\label{cor:Sheaf-Nis}
A presheaf $\sF$ of sets on the site $\Sch^G_{k/{\rm Nis}}$ 
(or $\Sm^G_{k/{\rm Nis}}$) is a sheaf if and only if $\sF(\0) = \star$ and it 
takes a square of the form ~\eqref{eqn:cd-square} to a Cartesian square.
\end{cor}
\begin{proof}
This is an immediate consequence of Theorem~\ref{thm:Comparison} and
\cite[Lemma~2.9, Proposition~2.15]{Voev1}.
\end{proof}

\begin{cor}\label{cor:Nis-Dim}
For any sheaf $\sF$ of abelian groups on the site $\Sch^G_{k/{\rm Nis}}$, one has
$H^i_{eN}(X, \sF) = 0$ for $i > \dim(X)$.
\end{cor}
\begin{proof}
This follows immediately from Proposition~\ref{prop:CRB},  
Theorem~\ref{thm:Comparison} and \cite[Theorem~2.7]{Voev1}.
\end{proof}

\begin{cor}\label{cor:sub-can}
The $eN$-topology on $\Sch^G_k$ is sub-canonical.
\end{cor}
\begin{proof}
Let $U \in \Sch^G_k$ and let us consider a square of the form 
~\eqref{eqn:cd-square}. By corollary~\ref{cor:Sheaf-Nis}, it suffices to show
that this square is Cartesian after applying the functor 
$\Hom_{\Sch^G_k}(-, U)$. So let $f_1 \in \Hom_{\Sch^G_{k}}(Y, U)$
and $f_2 \in \Hom_{\Sch^G_{k}}(A, U)$ be such that their restriction to
$B$ coincide. 

Since the $eN$-topology on $\Sch^G_k$ is known to be sub-canonical
for $G$ trivial, we find a unique $f \in \Hom_{\Sch_{k}}(X, U)$
such that $f \circ p = f_1$ and $f \circ j = f_2$. 
It remains to show that $f$ is $G$-equivariant.
Since the map
$p^{-1}(X \setminus A) \to X \setminus A$ is a $G$-equivariant isomorphism,
we see that the restrictions of $f$ to the $G$-invariant subsets 
$A$ and $X \setminus A$ are 
$G$-equivariant. It follows that $f$ is $G$-equivariant.
\end{proof}

More applications of Theorem~\ref{thm:Comparison} will appear
\S~\ref{subsection:LPMS}.

\vskip .3cm

\subsection{Points in the $eN$-topology}
\label{subsection:Points}
Recall that a {\sl point} $x$ on a Grothendieck site $\sC$ is a functor $x^*: \Shv(\sC)\to\Sets$ 
which commutes with all small colimits and finite limits. 
Such a functor acquires a right adjoint $x_*:\Sets\to\Shv(\sC)$ by Freyd's adjoint functor theorem.
Having enough points is convenient for expressing weak equivalences in the homotopy theory of 
simplicial presheaves on a site. 
Below we describe a set of points on the $eN$-site of $G$-schemes for $G$  a finite constant group scheme.

Given $X \in \Sch^G_k$ and $x \in X$, let $Gx$ denote the set-theoretic
$G$-orbit of $x$. Let $\sO^h_{X, Gx}$ denote the henselization
of the semi-local ring $\sO_{X, Gx}$ along the ideal defining the scheme $Gx$.
Set $X^h_{Gx} = \Spec(\sO^h_{X, Gx})$. One observes that the pair
$(X^h_{Gx}, Gx)$ is nothing but the filtering limit of all Nisnevich 
neighborhoods $(U, Gx)$ of $(X, Gx)$.
Since every Nisnevich neighborhood of $(X, Gx)$ contains an
$eN$-neighborhood (see \S~\ref{subsection:Nis-N}), we see that
$X^h_{Gx}$ is the filtered limit of all $eN$-neighborhoods of $(X, Gx)$. 
In particular, it acquires a canonical $Gx$-preserving compatible $G$-action. 

Given a pair $\underline{x} = (X, Gx)$, one gets a functor
$\underline{x}: \Sch^G_k \to \Sets$ by setting
$\underline{x}(U) = \Hom_{{\underset{\longleftarrow}{\Sch}}^G_k}(X^h_{Gx}, U)$.
Here, ${\underset{\longleftarrow}{\Sch}}^G_k$ denotes the category 
of all $k$-schemes with $G$-action (not necessarily of finite type). 
The left Kan extension of
this gives a functor $\underline{x}: \PShv^G_{\Sch_k} \to\Sets$
and one checks at once that its restriction to the subcategory $\Shv^G_{\Sch_k}$
indeed gives a point on $\Sch^G_{k/{\rm Nis}}$. 
We shall write this functor on presheaves as $F \mapsto F(X^h_{Gx})$.

\begin{prop}\label{prop:Point-Cons}
The collection $\{\underline{x}| X \in \Sch^G_k, x \in X\}$
is a conservative family of points on the site $\Sch^G_{k/{\rm Nis}}$.
\end{prop} 
\begin{proof}
By \cite[Proposition~6.5.a]{GV}, it is enough to show that if $U \in \Sch^G_k$
and if $\{f_i: U_i \to U\}_{i \in I}$ is a family of $G$-equivariant 
maps such that $\{\underline{x}(U_i) \to \underline{x}(U)\}_{i \in I}$ is
surjective for all $X \in \Sch^G_k$ and all $x \in X$, then 
$\{f_i\}$ is dominated by an $eN$-cover of $U$.

So suppose that $\{\underline{x}(U_i) \to \underline{x}(U)\}_{i \in I}$ 
is a surjective family for all pairs $(X, Gx)$. Let $u \in U$ and let
$v: (U^h_{Gu}, Gu) \to (U, Gu)$ be the resulting $G$-equivariant map.
By our assumption, we get an $i \in I$ and a $G$-equivariant factorization
\begin{equation}\label{eqn:Point*-1}
\xymatrix@C1.2pc{
& U_i \ar[d]^{f_i} \\
U^h_{Gu} \ar[ur]^{w}\ar[r]_{v} & U.}
\end{equation}
 
Notice that $w$ has to be an isomorphism on $Gu$ and hence gives a section
of $f_i$ along $Gu$. Since $(U^h_{Gu}, Gu)$ is the filtered limit of
$eN$-neighborhoods of $Gu$ and
since $f_i$ is a $G$-equivariant finite type morphism, we conclude that
there is an $eN$-neighborhood $(U'_i, Gu)$ and a 
$G$-equivariant factorization $(U'_i, Gu) \xrightarrow{w} (U_i, Gu)
\xrightarrow{f_i} (U, Gu)$. Since $u \in U$ was chosen as an arbitrary point,   
we get the desired domination of $\{f_i\}$.
\end{proof}

\vskip .3cm

\section{Model structures on simplicial presheaves on 
$\Sm^G_{k/{\rm Nis}}$}\label{section:Models}
Let $\sS$ denote the category of simplicial sets with internal hom objects
$\sS(-,-)$ defined, for example, in \cite[I.5]{GoerssJardine}. The category
of pointed simplicial sets will be denoted by $\sS_{\bullet}$. We have the 
pointed version of internal hom as well. A motivic $G$-space is a
contravariant functor $\Sm^G_k \to \sS$ and a pointed motivic $G$-space is a
contravariant functor $\Sm^G_k \to \sS_{\bullet}$. Due to the finite type
condition on $G$-schemes, the category $\Sm^G_k$ is essentially small, i.e.,
it is locally small with a small set of isomorphism classes of objects.
Let $\sM^G_k$ (resp. $\sM^G_{k, \bullet}$) denote the category of 
motivic (resp. pointed motivic) $G$-spaces. 
We may identify $\sS$ with the full subcategory of $\sM^G_k$ comprised of
constant motivic $G$-spaces. The Yoneda lemma yields a fully faithful
embedding of $\Sm^G_k$ into $\sM^G_k$ by sending $X \in \Sm^G_k$ to the
representable motivic $G$-space $h^G_X = \Hom_{\Sm^G_k}(-, X)$ taking values in
simplicial sets of dimension zero. Recall from Corollary~\ref{cor:sub-can}
that $h^G_X$ is a sheaf in the $eN$-topology. 
A pointed motivic $G$-space is just a motivic $G$-space $A$ with
a map $pt = h^G_k \to A$. In particular, a pointed $G$-scheme $(X,x)$
amounts to a $G$-scheme $X$ together with a $k$-rational $G$-fixed point
$x \in X$. In the following, we make no notational distinction between  
$X$ and $h^G_X$. For $X \in \Sm^G_k$, the symbol $X_{+}$ will denote the
pointed motivic $G$-scheme $(X \coprod pt, pt)$. 
We note the following useful fact about $\sM^G_{k,\bullet}$.

\begin{lem}\label{lem:elementary}
The category $\sM^G_{k, \bullet}$ is both a closed symmetric monoidal category
and a locally finitely presented bicomplete $\sS_{\bullet}$-category. 
In particular, filtered colimits commute with finite limits.
\end{lem}

The tensor product in $\sM^G_{k, \bullet}$ is defined by taking pointwise 
(schemewise) smash product
$(\sX \wedge \sY)(U) = \sX(U) \wedge \sY(U)$. With this definition,
$S^0 = pt \coprod pt = \Spec(k) \coprod \Spec(k)$ 
is the unit of the product and the limits, colimits are
defined pointwise. The functor ${\rm Ev}_U$ evaluating motivic $G$-spaces
at a fixed $G$-scheme $U$ is strict symmetric monoidal, preserves limits and
colimits, and there is an adjunction:
\begin{equation}\label{eqn:adjunc}
\xymatrix{
{\rm Fr}_U : 
\sS_{\bullet} 
\ar@<+.7ex>[r] &
\ar@<+.7ex>[l] \sM^G_{k, \bullet}: 
{\rm Ev}_U.
}
\end{equation}
The left adjoint ${\rm Fr}_U$, defined by ${\rm Fr}_U(K) = U_{+} \wedge K$, is
lax symmetric monoidal for any $G$-scheme and strict symmetric monoidal
when $U = pt$. For any $\sX \in \sM^G_{k, \bullet}$ and $K \in \sS_{\bullet}$,
we define $\sX \wedge K$ and $\sX^K$ by sending $U$ to 
$\sX(U) \wedge K$ and $\sS_{\bullet}(K, \sX(U))$, respectively.  

The $\sS_{\bullet}$-enrichment of motivic $G$-spaces is given degreewise by the
pointed simplicial set
\begin{equation}\label{eqn:simp-str}
{\sS(\sX, \sY)}_{n} = \Hom_{\sM^G_{k, \bullet}}(\sX \wedge \Delta[n]_{+}, \sY).
\end{equation}
The internal hom in $\sM^G_{k, \bullet}$ is defined pointwise as
$\sHom(\sX, \sY)(U) = \sS(\sX \wedge U_{+}, \sY)$. 

\vskip .3cm

A motivic $G$-space $\sX$ is finitely presentable if 
$\Hom_{\sM^G_{S,\bullet}}(\sX, -)$ commutes with filtered colimits. 
Using the natural isomorphism 
$\sHom(U_{+} \wedge K, \sX) \simeq \sX(U \times -)^K$,
one deduces that $\sX$ is finitely presentable if and only if $\sS(\sX, -)$ 
commutes with filtered colimits. 
The pointed finite simplicial sets and the $G$-schemes form the building 
blocks for $\sM^G_{S,\bullet}$ in the following sense
(see \cite[5.2.2b, 5.2.5]{handbook2}):

\begin{lem}\label{lem:finitecolimit}
Every pointed motivic $G$-space is a filtered colimit of finite colimits
of pointed motivic $G$-spaces of the form $(U \times \Delta[n])_{+}$,
where $U \in \Sm^G_k$ and $\Delta[n]$ is the standard $n$-simplex for $n \ge 0$.
The motivic $G$-spaces $(U \times \Delta[n])_{+}$ are finitely presented.
The finitely presented motivic $G$-spaces are closed under retracts,
finite colimits and tensor product.
\end{lem}  

In the above we described the monoidal structure on pointed motivic $G$-spaces. 
This story works verbatim for motivic $G$-spaces $\sM^G_S$ by replacing the smash 
product with the schemewise defined product $\sX \times \sY$.

\vskip .3cm

\subsection{Schemewise model structures}\label{subsection:SMS}
The goal of this section is to construct model structures on
motivic $G$-spaces. We first describe these model structures for the
unpointed motivic $G$-spaces and show in the end how these model structures
induce such structures on the pointed motivic $G$-spaces.
We refer the reader to \cite{Hirsc} for standard notions related to
model structures. We only recall here that a model structure on $\sM^G_k$
is a simplicial model structure if the simplicial structure interacts with
cofibrations, fibrations and weak equivalences: 
If $i :\sX \to \sY$ is a cofibration and $p : \sZ \to \sW$ is a fibration in 
$\sM^G_k$, then the map of simplicial sets
\[
\sS(\sY, \sZ) \xrightarrow{(i^*, p_*)}
\sS(\sX,\sZ) {\underset{\sS(\sX,\sW)}\times} \
\sS(\sY, \sW)
\]
is a Kan fibration, which is a weak equivalence if either $i$ or $p$ is a weak
equivalence.

\vskip .2cm

We shall say that a map $f : \sX \to \sY$ of motivic $G$-spaces is a 
{\sl schemewise weak equivalence} (resp.~{\sl schemewise fibration}) if the map of simplicial sets
$\sX(X) \to \sY(X)$ is a weak equivalence (resp. Kan fibration) of simplicial 
sets for every $X \in \Sm^G_k$. 
Moreover, $f$ is called a projective cofibration if it has the left 
lifting property with respect to all maps which are schemewise fibrations and
weak equivalences. It follows from
\cite[Theorems~11.6.1, 11.7.3, 13.1.14, Proposition~12.1.5 ]{Hirsc}
that $\sM^G_k$ acquires the so-called {\sl projective model structure}:

\begin{thm}$($Projective model structure$)$\label{thm:PMS-Psh} 
The schemewise fibrations and weak equivalence, and projective cofibrations 
form a combinatorial and simplicial model structure on $\sM^G_k$ with 
respect to the $\sS$-enrichment in ~\eqref{eqn:simp-str}.

The set of generating cofibrations
\[
I^{\rm sch}_{\rm proj}(\sm^G_k) = \{U_{+} \wedge (\partial \Delta^n \subsetneq
\Delta^n)_{+}\}_{n \ge 0, U \in \Sm^G_k}
\]
and trivial cofibrations
\[
J^{\rm sch}_{\rm proj}(\sm^G_k) = \{U_{+} \wedge (\Lambda^n_i \subsetneq
\Delta^n)_{+}\}_{n \ge 0, 0 \le i \le n, U \in \Sm^G_k}
\]
are induced from the corresponding maps in $\sS$. The domains and
codomains of the maps in these generating sets are finitely presented.
The projective model structure is proper. For every $U \in \Sm^G_k$,
the pair $({\rm Fr}_U, {\rm Ev}_U)$ forms a Quillen pair.
\end{thm}

Let $\kappa$ be the first cardinal number greater than the cardinality of the
set of maps in $\PShv^G_{\Sm_k}$. 
If $\omega$ denotes, as usual, the cardinal of
continuum, we define $\gamma$ as ${\kappa \omega}^{\kappa \omega}$. Now let
$I^{\rm sch, \kappa}_{\rm inj}(\sm^G_k)$ be the set of maps $\sX \to \sY$ such that
$\sX(U) \to \sY(U)$ is a cofibration of simplicial sets of cardinality 
less than $\kappa$ for every $U \in \Sm^G_k$. Likewise, we define
$J^{\rm sch, \gamma}_{\rm inj}(\sm^G_k)$ for schemewise trivial cofibrations 
of simplicial sets bounded by $\gamma$. With these definitions, the 
following holds for the so-called {\sl injective model structure} on
$\sM^G_k$, see \cite{Heller}, \cite{Joyal}.

\begin{thm}$($Injective model structure$)$\label{thm:LMS-Psh} 
There is a cofibrantly generated model structure on $\sM^G_k$ with
schemewise cofibrations and weak equivalences, and injective fibrations.
The cofibrations and trivial cofibrations are generated by
$I^{\rm sch, \kappa}_{\rm inj}(\sm^G_k)$ and $J^{\rm sch, \gamma}_{\rm inj}(\sm^G_k)$,
respectively. The injective model structure is combinatorial, proper and 
simplicial with the $\sS$-enrichment in ~\eqref{eqn:simp-str}. 
\end{thm}

The third model structure one can consider is an example of a so-called
flasque model structure \cite{Isak}. It is obtained by considering 
equivariant embeddings of smooth $G$-subschemes, generalizing the
cognate schemewise model structure  in \cite[Theorem~A.9]{PPR} for the trivial
group. For $U \in \Sm^G_k$, we consider a finite set of $G$-equivariant  
monomorphisms $V_I = \{V_i \to U\}_{i \in I}$. The categorical union 
${\underset{i \in I}\cup} V_i$ is the coequalizer of the diagram
\[
\xymatrix{
{\underset{i,j \in I}\coprod} \ V_i {\underset{U}\times} \ V_j
\ar@<+.7ex>[r] 
\ar@<-.7ex>[r] &
{\underset{i \in I}\coprod} \ V_i
}
\]
formed in $\sM^G_k$. We denote by $i_I$ the induced monomorphism
${\underset{i \in I}\cup} V_i \to U$. Note that $\0 \to U$ arises in this
way. The push-out product of maps of $i_I$ and a map between simplicial sets
exists in $\sM^G_k$. In particular, we are entitled to form the sets
\[
I^{\rm sch}_{\rm clo}(\sm^G_k) = 
\{i_I  \ \square \ {(\partial \Delta^n \subset \Delta^n)}_{+}\}_{I, n \ge 0}
\]
and
\[
J^{\rm sch}_{\rm clo}(\sm^G_k) = 
\{i_I \ \square \ (\Lambda^n_i \subset \Delta^n)_{+}\}_{I, n \ge 0, 0 \le i \le n}.
\]
A map between motivic $G$-spaces is called a closed schemewise fibration if
it has the right lifting property with respect to $J^{\rm sch}_{\rm clo}(\sm^G_k)$.
A closed schemewise cofibration is a map having the left lifting property
with respect to every trivial closed schemewise fibration.

\begin{thm}$($Flasque model structure$)$\label{thm:PMS-Flasque}
The schemewise weak equivalences, closed schemewise cofibrations and
fibrations form a combinatorial and simplicial model structure on $\sM^G_k$
with respect to the $\sS$-enrichment in ~\eqref{eqn:simp-str}.
The closed schemewise cofibrations and fibrations are generated by
$I^{\rm sch}_{\rm clo}(\sm^G_k)$ and $J^{\rm sch}_{\rm clo}(\sm^G_k)$, respectively.
Moreover, the flasque model structure is cellular and proper.
\end{thm}

\vskip .3cm

\subsection{Local model structures}\label{subsection:LPMS}
Recall from \cite[Chapter~3]{Hirsc} that if $\Sigma$ is a class of
morphisms in a simplicial model structure on $\sM^G_k$, then an object $\sZ$ 
of $\sM^G_k$ is called $\Sigma$-local, if it is fibrant and for every element 
$f : \sX \to \sY$ in $\Sigma$, the induced map of homotopy function complexes 
$\sS(\sY, \sZ) \to \sS(\sX, \sZ)$
is a weak equivalence (see \cite[Definitions~3.1.4, 17.1.1]{Hirsc}).
Moreover, a map $f : \sX \to \sY$ in $\sM^G_k$ is a $\Sigma$-local 
equivalence if for every $\Sigma$-local object $\sZ$, the induced map
of homotopy function complexes $\sS(\sY, \sZ) \to \sS(\sX, \sZ)$ 
is a weak equivalence. Clearly every element of $\Sigma$ is a $\Sigma$-local
equivalence.

The {\sl left Bousfield localization} of $\sM^G_k$ with respect
to $\Sigma$ is a model category structure $L_{\Sigma}\sM^G_k$ on the
underlying category $\sM^G_k$ such that 
\begin{enumerate}
\item
weak equivalences coincide with the $\Sigma$-local equivalences of $\sM^G_k$, 
\item
cofibrations coincide with the cofibrations of $\sM^G_k$, and
\item
fibrations coincide with the maps having the
right lifting property with respect to cofibrations that are simultaneously
$\Sigma$-local equivalences.
\end{enumerate}

We shall employ the technique of Bousfield localization to define local
model structures on $\sM^G_k$. One basic idea underlying the local model
structures is that the distinguished $eN$-squares inform our definition
of locally fibrant motivic $G$-spaces, and hence the accompanying 
(co)homology theories on $\Sm^G_k$.

\begin{defn}\label{defn:locallyfibrant}
A motivic $G$-space $\sX$ is called locally projective fibrant if it is
schemewise fibrant and flasque; i.e., $\sX(\0)$ is contractible and
the square
\begin{equation}\label{eqn:flasqueness}
\xymatrix@C1.3pc{
\sX(X)  \ar[r]^{\sX(j)} \ar[d]_{\sX(p)} & \sX(A) \ar[d] \\
\sX(Y) \ar[r] & X(B)} 
\end{equation}
is homotopy Cartesian for every distinguished $eN$-square of the 
form ~\eqref{eqn:cd-square}.

The locally injective fibrant and locally flasque fibrant motivic $G$-spaces
are defined analogously by means of schemewise injective and flasque
model structures, respectively.
\end{defn}

Let $(-)^{\rm cof}: \sM^G_k \to \sM^G_k$ be a cofibrant replacement functor
in the schemewise projective model structure.

\begin{defn}\label{defn:local-fibrant-1}
A map $\sX \to \sY$ of motivic $G$-spaces is called a local projective
weak equivalence if the induced map
\[
\sS(\sY^{\rm cof}, \sZ) \to \sS(\sX^{\rm cof}, \sZ)
\]
is a weak equivalence for every locally projective fibrant motivic $G$-space
$\sZ$. A map is a local projective fibration if it has the right lifting
property with respect to projective cofibrations which are simultaneously
local projective weak equivalence. The local injective and local flasque
weak equivalences and fibrations are defined analogously.
\end{defn}

\begin{thm}\label{thm:LFMS-Psh}
The category $\sM^G_k$ acquires local projective, injective and flasque
model structures. All of these model structures are combinatorial, proper
and simplicial. The identity functors from the local projective model 
structure to the local flasque and local injective model structures are
left Quillen equivalences. 
\end{thm}
\begin{proof}
The schemewise model structures are combinatorial and left proper ones,
and hence suitable fodder for Bousfield localizations $L_{\Sigma}M^G_k$,
where we define $\Sigma$ by means of distinguished $eN$-squares. In order
to identify these Bousfield localizations with the definitions above, we
shall make repeated use of the fact that the cofibrations and the fibrant
objects determine the weak equivalences in any model structure.
The existence of the model structures follows by reconciling the locally
fibrant motivic $G$-spaces in the sense of Definition~\ref{defn:locallyfibrant}
with the fibrant objects in the Bousfield localizations determined by $\Sigma$.
Once we do these identifications, the claim about the simplicial and the
left properness property follows because these properties are preserved 
under Bousfield localization (see \cite[Theorem~4.1.1]{Hirsc}).

We start by defining $\Sigma$ in the case of the local projective model
structure. For a distinguished $eN$-square $Q$ as in ~\eqref{eqn:cd-square}, 
let $Q^{hp}$ be the homotopy push-out of $A \leftarrow B \rightarrow Y$ in the
schemewise projective model structure. There is a canonical map
$Q^{hp} \to X$ and we set
\begin{equation}\label{eqn:Ghp-local}
\Sigma^{hp}_{\rm Nis} = \{Q^{hp} \to X\}_Q \cup \{* \to \0_{+}\}.
\end{equation}
In the case of the local injective and flasque model structures, we consider
the categorical push-out $Q^p$ of $A \leftarrow B \rightarrow Y$ in
$\sM^G_k$. There is a canonical map $Q^p \to X$ and we set
\begin{equation}\label{eqn:G-local}
\Sigma^{p}_{\rm Nis} = \{Q^{p} \to X\}_Q \cup \{* \to \0_{+}\}.
\end{equation}

We claim that the fibrant objects in $L_{\Sigma^{hp}_{\rm Nis}}$ coincide with the
local projective fibrant objects introduced in 
Definition~\ref{defn:locallyfibrant}. In effect, an object in the 
localization $L_{\Sigma^{hp}_{\rm Nis}}$ is fibrant if and only if it is 
schemewise fibrant and $\Sigma^{hp}_{\rm Nis}$-local. 
But by adjunction, this is same
as saying that it takes a distinguished $eN$-square to a homotopy
Cartesian square and this in turn is same as saying that it is locally 
projective fibrant. The right properness of the local projective model
structure $L_{\Sigma^{hp}_{\rm Nis}}$ follows from \cite[Theorem~1.5]{Blander}.
There is a parallel story for the injective and flasque model structures.
In these cases, $B \to Y$ is a cofibration so that the categorical 
push-out $Q^p$
is a model for the homotopy push-out. For this reason, it suffices to consider
the set $\Sigma^p_{\rm Nis}$ when constructing the local injective and flasque
model structures.

We also observe that the weak equivalences in the local projective,
injective and flasque model structures are same. 
It follows from Lemma~\ref{lem:PF-fib} that a map which is either a 
local injective fibration or a local flasque fibration, is also a 
local projective fibration. We conclude from this that
that the local injective and the local flasque model 
structures are also right proper.

By \cite[Theorems~2.2, 3.7]{Isak}, it follows that the identity functors from 
the schemewise projective model structure to schemewise flasque and
the schemewise injective model structures on $\sM^G_k$ are left Quillen 
equivalences. We have just shown that the local 
projective, local injective and local flasque model structures 
are obtained by the Bousfield localizations of the corresponding schemewise
model structures with respect to the same set.
The second part of the theorem now follows from \cite[Theorem~3.3.20]{Hirsc}.
This completes the proof.
\end{proof}

\begin{lem}\label{lem:PF-fib}
Let $\Sigma$ be a set of maps in $\sM^G_k$. 
Suppose $f :\sX \to \sY$ is a fibration in the Bousfield localization  
$L_{\Sigma}\sM^G_k$ with respect to the schemewise injective model 
structure or the schemewise flasque model structure.  
Then $f$ is a fibration in the Bousfield localization with respect 
to the schemewise projective model structure.
\end{lem}
\begin{proof}
Every schemewise projective cofibration is also a schemewise injective and 
flasque cofibration, cf.~ \cite[Proposition~11.6.3.]{Hirsc}, 
\cite[Theorem~3.7]{Isak}. 
The result follows now since the weak equivalences in the Bousfield localized 
model structures on $L_{\Sigma}\sM^G_S$ coincide,
cf.~the proof of Theorem \ref{thm:LFMS-Psh}.
\end{proof}

Combining Theorems~\ref{thm:Comparison} and \ref{thm:LFMS-Psh}, we get
the following explicit description of the weak equivalences in the 
local projective, injective and flasque model structures on $\sM^G_k$
when $G$ is a finite constant group scheme over $k$.
This description
is closest to the description of local weak equivalence
of simplicial presheaves in the non-equivariant Nisnevich topology and 
reflects our usage of the $eN$-topology.

\begin{thm}\label{thm:Local-we}
Assume that $G$ is a finite constant group scheme over $k$.
A map $f: \sX \to \sY$ in $\sM^G_k$ is a weak equivalence in the local
projective, injective, and flasque model structures if and only if
for all $X \in \Sm^G_k$ and all $x \in X$, the map of simplicial sets
$\sX(X^h_{Gx}) \to \sY(X^h_{Gx})$ (see \S~\ref{subsection:Points}) is a weak 
equivalence.
\end{thm}
\begin{proof}
It follows from Theorems~\ref{thm:Comparison} and ~\ref{thm:LFMS-Psh} and 
\cite[Theorem~3.8]{Voev1} that $f$ is a weak equivalence in the local
projective, injective and flasque model structure if and only if
the following hold.
\begin{enumerate}
\item
The map $f_* : \pi_0(\sX) \to \pi_0(\sY)$ induces isomorphism
of the associated sheaves.
\item
For all $X \in \Sm^G_k$, all choices of base points $x \in \sX(X)_0$ and
all $n \ge 1$, the map $f_* : \pi_n(\sX, x) \to \pi_n(\sY, f(x))$ 
induces an isomorphism of the associated $eN$-sheaves on the 
site $\Sm^G_k \downarrow X$.
\end{enumerate} 

But this is same as saying that for all points $x^*:\sM^G_k \to \sS$ of the 
$eN$-site $\Sm^G_{k/Nis}$, the map $f_*: x^*(\sX) \to x^*(\sY)$ is a weak 
equivalence. Recall here that the $eN$-topology on $\Sm^G_k$ is sub-canonical 
(Corollary~\ref{cor:sub-can}) and hence every point 
$x^*: \Shv^G_{\Sm_k} \to Sets$ (of the site $\Sm^G_{k/Nis}$) has the 
left Kan extension to a functor $x^*: \sM^G_k \to \sS$.

Now, it follows from \cite[Remark~2.1.3]{MV} that $f: \sX \to \sY$ is a 
local weak equivalence if and only if $f_*: x^*(\sX) \to x^*(\sY)$ is a weak
equivalence for all $x$ lying in a conservative family of points of 
$\Sm^G_{k/Nis}$.
The theorem now follows by applying Proposition~\ref{prop:Point-Cons}.
\end{proof}

The following result is another consequence of Theorem~\ref{thm:Comparison}.
A refined version, see Theorem~\ref{thm:A1-flasq},
will be used to prove representability 
of equivariant $K$-theory in the equivariant motivic homotopy 
category.

\begin{prop}\label{prop:Flasq-fib}
Let $\sX$ be a motivic $G$-space and $\sX \to \wh{\sX}$ a fibrant replacement 
in the local projective model structure on $\sM^G_k$. 
Then $\sX$ is flasque if and only if the map $\sX \to \wh{\sX}$ is a 
schemewise weak equivalence. 
The same result holds for the local injective and local flasque model 
structures on $\sM^G_k$. 
\end{prop}
\begin{proof}
First suppose that $\wh{\sX}$ is a fibrant replacement of $\sX$ in the local projective model structure on $\sM^G_S$.
It follows from Theorem~\ref{thm:Comparison} and \cite[Lemma~4.1]{Blander} that $\wh{\sX}$ is a flasque presheaf.
 
If $\wh{\sX}$ is a fibrant replacement of $\sX$ in the local injective or flasque model structure, 
then it follows from Lemma~\ref{lem:PF-fib} that it is a fibrant replacement of $\sX$ in the local projective model structure too. 
Hence $\wh{\sX}$ is flasque as shown above.

Suppose now that $\sX$ is flasque. 
Theorem~\ref{thm:Comparison} and \cite[Lemma~3.5]{Voev1} imply that $\sX \to \wh{\sX}$ is a schemewise weak equivalence. 
The converse implication is trivial.
\end{proof}

\section{The equivariant motivic homotopy category 
${\rm Ho}^G_{\A^1}(k)$}\label{section:UHC}
In this section we construct the unstable homotopy category of motivic
$G$-spaces. This is done by the following $\A^1$-localization of our 
local model structures.

\subsection{$\A^1$-localization of $\sM^G_k$}
\label{section:A-Local}
Let $T$ be a site with category of presheaves ${\rm {\bf PSh}}(T)$.
Let $pt$ denote the terminal object of ${\rm {\bf PSh}}(T)$.
Recall from \cite[2.2.3]{MV} that an interval on a site $T$
is a presheaf $I \in {\rm {\bf PSh}}(T)$ together with morphisms:
\[ 
\mu : I \times I \to I ; \ \ \ i_0, i_1: pt \to I
\]
where $pt$ is the terminal object in $Psh(T)$ with 
the canonical morphism $p : I \to pt$ such that 
\[
\mu(i_0 \times {\rm id}_I) = \mu({\rm Id}_I \times i_0) = i_0 \circ p
\]
\[
\mu(i_1 \times {\rm id}_I) = \mu({\rm Id}_I \times i_1) = {\rm id}_I 
\]
and the morphism $i_0 \coprod i_1 : pt \coprod pt \to I$ is a monomorphism.

In what follows, we let $I = \A^1_k$ with trivial $G$-action and $pt =
\Spec(k)$ such that $i_0(s) = (s,0), i_1(s) = (s,1)$ and 
$\mu(a, b) = ab$. 
It is then immediate that the pair 
$\left(\Sm^G_{k/{\rm Nis}}, \A^1_k \right)$ is a site with interval. 
Since the base field $k$ is fixed throughout, we shall
write $\A^1$ for the affine line over $k$.

\begin{defn}\label{defn:A-1model-str}
The motivic projective (resp. injective, flasque) model structure on 
$\sM^G_k$ is the left Bousfield localization of its local 
projective (resp.~injective, flasque) structure with respect to the set of
projection maps 
\[\{X \times \A^1 \xrightarrow{p_X} X | \ X \in \Sm^G_k\}.\] 

The motivic $G$-spaces which are local with respect to this set of maps 
are called $\A^1$-local. The fibrant objects in the motivic projective
(resp. injective, flasque) model structure will be called $\A^1$-fibrant.
A weak equivalence in the motivic projective
(resp.~injective, flasque) model structure will be called a
{\sl motivic weak equivalence}. 
\end{defn}

In the motivic injective and flasque model structures, the $\A^1$-local
objects can be described using the following simpler criterion.
We say that a motivic $G$-space $\sX$ is {\sl $\A^1$-weak invariant} if
for all $X \in \Sm^G_k$, the naturally induced map 
$\sX(X) \to \sX(X \times \A^1)$ is a weak equivalence.

\begin{lem}\label{lem:A1-compare}
Suppose that a motivic $G$-space $\sX$ is fibrant in the local injective 
(or flasque) model structure. Then it is $\A^1$-fibrant if and only if it is 
$\A^1$-weak invariant. 
\end{lem}
\begin{proof}
We first observe that as $\sX$ is already locally fibrant, it is $\A^1$-fibrant
if and only if it is $\A^1$-local. The lemma is now a consequence  of 
\cite[Definition~3.1.4, Proposition~16.1.3]{Hirsc} using the observation that
every $X \in \Sm^G_k$ is cofibrant in the local injective and flasque
model structures. 
\end{proof}

\vskip .3cm

The motivic weak equivalences in $\sM^G_k$ are those maps which are 
$\Sigma^{hp}_{\rm Nis}$-local (resp. $\Sigma^p_{\rm Nis}$-local) and 
$\A^1$-local weak equivalences. 
The cofibrations coincide with the cofibrations of the underlying local model
structures and the fibrations are maps having the right
lifting property with respect to cofibrations which are simultaneously
motivic weak equivalences.

Theorem~\ref{thm:LFMS-Psh} and \cite[Theorem~4.1.1]{Hirsc} imply 
that the motivic projective, injective and flasque model structures on 
$\sM^G_k$ are left proper, cellular and simplicial. 
Moreover, right properness of the motivic model structures 
follows from \cite[Lemma~3.1]{Blander} and 
Lemma~\ref{lem:PF-fib}.

It follows from \cite[Theorem~3.3.20]{Hirsc} that the identity functors from
the motivic projective to the motivic flasque and injective model structures
are left Quillen equivalences. In particular, these  
model structures have equivalent homotopy categories, which will be denoted by 
${\rm Ho}^G_{\A^1}(k)$. Given motivic $G$-spaces $\sX$ and $\sY$,
the set $\Hom_{{\rm Ho}^G_{\A^1}(k)}(\sX, \sY)$ will be denoted by
$[\sX, \sY]_{G,\A^1}$.

\vskip .3cm

\subsubsection{$\A^1$-flasque sheaves}\label{subsubsection:A1FL}
We shall say that a motivic $G$-space is $\A^1$-flasque if it is flasque
and $\A^1$-weak invariant. As another application of 
Theorem~\ref{thm:Comparison}, we get the following extension of 
Proposition~\ref{prop:Flasq-fib} to the motivic model structures.
This result is very useful in determining the schemewise weak equivalences
of motivic $G$-spaces as demonstrated in our proof of representability 
of equivariant $K$-theory.

\begin{thm}\label{thm:A1-flasq}
A motivic $G$-space $\sX$ is $\A^1$-flasque if and only if every
fibrant replacement in the motivic injective (resp. flasque)
model structure is a schemewise weak equivalence.
A map $f: \sX \to \sY$ of $\A^1$-flasque motivic spaces is a
motivic weak equivalence if and only if it is a schemewise weak equivalence.
\end{thm}
\begin{proof}
The `if' part of the first assertion follows from 
Proposition~\ref{prop:Flasq-fib} and Lemma~\ref{lem:A1-compare}.
To prove the converse, suppose that $\sX$ is
an $\A^1$-flasque motivic $G$-space and let $f: \sX \to \wh{\sX}$ be 
an $\A^1$-fibrant replacement. By Proposition~\ref{prop:Flasq-fib},
it is enough to show that $f$ is also a locally (i.e., in the local
injective or flasque model structure) fibrant replacement. 

We factor $f$ as a composition $\sX \xrightarrow{g} \sX' \xrightarrow{f'} 
\wh{\sX}$, where $g$ is a local trivial cofibration 
(in particular, motivic trivial cofibration) and $f'$ is a local  
fibration. It follows from the 2-out-of-3 axiom that $f'$ is a motivic 
weak equivalence. We need to show that $f'$ is a local 
weak equivalence.

Since $\wh{\sX}$ is locally fibrant and $f'$ is a local fibration, it follows 
that $\sX'$ is locally fibrant. In particular, $g$ defines a locally  
fibrant replacement of $\sX$. We conclude from 
Proposition~\ref{prop:Flasq-fib} that $g$ is a schemewise weak equivalence.
We now apply the $\A^1$-weak invariance of $\sX$ and     
Lemma~\ref{lem:A1-compare} to conclude that $\sX'$ is $\A^1$-fibrant.
In particular, it is $\Sigma^p_{\rm Nis}$-local as well as $\A^1$-local.
We have thus shown that $f'$ is a motivic weak equivalence of 
$\A^1$-fibrant motivic $G$-spaces. It follows from the local
Whitehead theorem (see \cite[Theorem~3.2.12]{Hirsc}) that $f'$ is
in fact a schemewise weak equivalence. This proves the first part of the
theorem.

To prove the second assertion of the theorem for the motivic weak equivalence  
$f:\sX \to \sY$ of $\A^1$-flasque motivic $G$-spaces, we can form a
commutative diagram
\begin{equation}\label{eqn:A1fib}
\xymatrix{ \sX \ar[r]^f  \ar[d] & \sY \ar[d] \\ 
\wh{\sX} \ar[r]_{\wh{f}} & \wh{\sY}}
\end{equation}
where the vertical arrows are $\A^1$-fibrant replacements. 
It follows from the 
$2$-out-of-$3$ axiom that $\wh{f}$ is a motivic weak equivalence. In this case, 
$\wh{f}$ is a schemewise weak equivalence by \cite[Theorem~3.2.12]{Hirsc}. 
The two vertical arrows are also schemewise weak equivalences by the first 
assertion of the theorem. It follows that $f$ is a schemewise weak equivalence. 
\end{proof}

\vskip .3cm

\subsection{Equivariant vector bundles}
\label{subsubsection:EBM}
To justify the construction of the motivic model structure by inverting
the trivial line bundle $\A^1$, we show that this in fact makes all equivariant
vector bundle projections into motivic weak equivalences. 

For maps $f,g : \sX \to \sY$ of motivic $G$-spaces, an
elementary $\A^1$-homotopy from $f$ to $g$ is a morphism $H : \sX \times
\A^1 \to \sY$ such that $H \circ i_0 = f$ and $H \circ i_1 = g$. Two maps
are called equivariantly $\A^1$-homotopic if they can be connected by a sequence of 
elementary $\A^1$-homotopies. A map $f : \sX \to \sY$ is called a strict equivariant 
$\A^1$-homotopy equivalence if there is a morphism $g : \sY \to \sX$ such that
$f \circ g$ and $g \circ f$ are equivariantly $\A^1$-homotopic to the respective 
identity maps.

\begin{prop}\label{prop:Hom-inv}
Let $X \in \Sm^G_k$ and let $\sV \xrightarrow{f} X$ be a $G$-equivariant
vector bundle. Then the map of associated motivic $G$-spaces is a 
motivic weak equivalence.
\end{prop}
\begin{proof} 
Recall that the $eN$-site has an interval $I$ defined in the beginning of 
\S~\ref{section:A-Local}.
Moreover, it is clear that our equivariant motivic homotopy category is
obtained precisely by inverting the $I$-local morphisms in the
sense of \cite[\S~2.2.3]{MV}. 
Hence it follows from \cite[Lemma~2.3.6]{MV} that a strict 
$\A^1$-homotopy equivalence of motivic $G$-spaces is a motivic weak equivalence.
Thus it suffices to show that the map $f:\sV \to X$ is a strict 
$\A^1$-homotopy equivalence.

We can assume that $X$ is $G$-connected in the sense that $G(k)$ acts
transitively on the set of connected components of $X$. 
Suppose $\sV$ has rank $n$ and let $\sU = \{U_1, \cdots,
U_r\}$ be a Zariski open cover (not necessarily $G$-invariant) of $X$ such 
that each $U_i = \Spec(R_i)$ is affine and $\sV_i = f^{-1}(U_i) \to U_i$ is
a trivial ordinary bundle given by $\sV_i = 
\Spec \left(R_i[X^i_1, \cdots , X^i_n]\right)$.

Define the ring map $H: R_i[X^i_1, \cdots , X^i_n] \to
R_i[T, X^i_1, \cdots , X^i_n]$ by $X^i_j \mapsto TX^i_j$. It is 
straightforward to check that since these maps are natural once we fix the
$T$-coordinate over $X$, they glue together to give an elementary homotopy
\begin{equation}\label{eqn:Hom-inv1}
H : \sV \times \A^1 \to \sV
\end{equation}
such that $H \circ i_0 = i_X \circ f$ and $H \circ i_1 = {\rm id}_{\sV}$,
where $i_X : X \to \sV$ is the zero-section. Note that
$i_0$, $i_1$ and $i_X$ are all $G$-equivariant. Thus, we shall be done if we
show that $H$ is $G$-equivariant. 

Now $f$ is a $G$-equivariant vector bundle, so that over every point
$x \in X$, the fiber of $f$ is a $k_x$-vector space $\sV_x$ of rank $n$. 
Moreover, if $g \in G(k(x))$ is such that $gx = x'$, then $g$ acts on $\sV_x$ 
by a $k_x \ (\simeq k_{x'})$-linear isomorphism $\sV_x \rightarrow \sV_{x'}$.
At the level of the coordinate rings of these fibers, the $G$-action
and the map $H$ are described by the diagram
\begin{equation}\label{eqn:bundles}
\xymatrix@C1pc{
k(x)[x_1, \cdots , x_n] \ar[r]^{\tau_g} \ar[d]_{H} & 
k(x')[x'_1, \cdots , x'_n] \ar[d]^{H} \\
k(x)[t, x_1, \cdots , x_n] \ar[r]_{\tau_g} & 
k(x')[t, x'_1, \cdots , x'_n],}
\end{equation}
where $\tau_g$ is the map on the coordinate rings induced by $g \in G(k(x))$.
It is straightforward to check that this diagram commutes, which shows that
$H$ is $G$-equivariant. 
\end{proof}

\vskip .3cm

\section{The equivariant motivic homotopy category 
${\rm Ho}^G_{\A^1, \bullet}(k)$}\label{section:UHC-P}
Recall from \S~\ref{section:Models} the category $\sM^G_{k,\bullet}$ of pointed 
motivic $G$-spaces.
Lemma~\ref{lem:elementary} shows that $\sM^G_{k,\bullet}$ is closed symmetric 
monoidal with respect to the smash product and pointed internal homs. 
There is an adjoint functor pair 
\[
\xymatrix{
\sM^G_k
\ar@<+.7ex>[r] &
\ar@<+.7ex>[l]
\sM^G_{k,\bullet}
}
\]
where the left adjoint adjoins a disjoint base point, 
$\sX\mapsto\sX_+=(\sX \coprod pt, pt)$ and the right adjoint is the forgetful functor.
Since $\sM^G_{k,\bullet}$ is the slice category $pt \downarrow \sM^G_k$, we 
conclude the existence of the following {\sl motivic injective model 
structure} from \cite[Theorem~7.6.5]{Hirsc}.

\begin{thm}\label{thm:Pntd-MS}
The category $\sM^G_{k, \bullet}$ admits a model
structure where a map $f : (\sX, x) \to (\sY, y)$ is a weak equivalence
(resp. cofibration, resp. fibration) if and only if
$f : \sX \to \sY$ is a weak equivalence (resp. cofibration, resp. fibration)
in the motivic injective model structure ({\sl cf.}
Definition~\ref{defn:A-1model-str}) after applying the forgetful functor.  
This model structure is proper, cellular and simplicial.
\end{thm}

The motivic projective model structure on $\sM^G_{k,\bullet}$ is defined by 
replacing the local injective model structure in Theorem~\ref{thm:Pntd-MS} 
by the local projective model structure. 
Likewise for the motivic flasque model structure. 
As in the unpointed case,
the three model structures are Quillen equivalent and hence have equivalent 
homotopy categories,
which justifies the following definition.

\begin{defn}\label{defn:EUHC}
The equivariant pointed motivic homotopy category 
${\rm Ho}^G_{\A^1, \bullet}(k)$ is the homotopy category of pointed motivic 
$G$-spaces with 
respect to either of the motivic model structures. 
For pointed motivic $G$-spaces $\sX$ and $\sY$, we let 
$[\sX,\sY]_{G,\A^1_{\bullet}}$ denote the set 
$\Hom_{{\rm Ho}^G_{\A^1,\bullet}(k)}(\sX, \sY)$.
Let ${\rm Ho}^G_{eN, \bullet}(k)$ denote the homotopy category of pointed 
motivic $G$-spaces with respect to either of the local model structures.
\end{defn}

\begin{prop}\label{prop:Ho-Smash}
The smash product preserves weak equivalences and cofibrations in the
motivic injective model structure on $\sM^G_{k, \bullet}$. 
This induces a symmetric closed monoidal category structure on
${\rm Ho}^G_{\A^1, \bullet}(k)$.
\end{prop}
\begin{proof}
Since the weak equivalences in the motivic projective and injective
model structures are same, it follows from \cite[Lemma~2.20]{DRO}
that smashing with any pointed motivic $G$-space
preserves motivic weak equivalence.
Since the cofibrations in the motivic (injective) model structure are 
monomorphisms, it follows immediately that smash product preserves cofibrations.

The first assertion implies that the smash product defines a structure
of symmetric monoidal structure on ${\rm Ho}^G_{\A^1, \bullet}(k)$.
We need to show that this monoidal structure is closed to complete
the proof. Since the motivic projective and injective model structures
have equivalent homotopy categories, it suffices to show that
the motivic projective model structure on $\sM^G_{k, \bullet}$
is monoidal. But this follows from \cite[Corollary~2.19]{DRO}.
\end{proof}

Recall that the simplicial circle $S^1_s$ is the constant presheaf
${\Delta[1]}/{\partial \Delta[1]}$ pointed by the image of
$\partial \Delta[1]$. We shall write $(S^1_s)^{{\wedge} n}$ as $S^n_s$.
Smashing with the simplicial circle gives a functor
\[
\Sigma_s(\sF, x) = S^1_s \wedge (\sF, x).
\]
Let $\Omega^1_s(-) = \sHom_{\bullet}(S^1_s, -)$ be the right adjoint  
of $S^1_s \wedge (-)$.
Proposition~\ref{prop:Ho-Smash} implies that 
$\left(\Sigma_s(-), \Omega^1_s(-)\right)$ is a Quillen pair of endofunctors
on $\sM^G_{k, \bullet}$. In particular, we get an
adjoint pair of endofunctors
\begin{equation}\label{eqn:smash-hom}
\xymatrix{
\Sigma_s(-):
{\rm Ho}^G_{\A^1, \bullet}(k)
\ar@<+.7ex>[r] &
\ar@<+.7ex>[l]
{\rm Ho}^G_{\A^1, \bullet}(k) : {\bf R}\Omega^1_s(-).
}
\end{equation}
The functor ${\bf R}\Omega^1_s \left((\sX, x)\right)$ is given as 
$\Omega^1_s \left(Ex \left((\sX, x)\right)\right)$, where
$Ex \left((\sX, x)\right)$ is a cofibrant fibrant replacement of 
$(\sX, x)$ in the motivic model structure.

\vskip .3cm

\subsection{Equivariant motivic homotopy groups}\label{subsection:EMHG}
We end this section with the definition of equivariant motivic homotopy
groups of motivic $G$-spaces and show that these groups coincide with
the actual homotopy groups of an $\A^1$-fibrant replacement.
The results of this section will be used in proving representability of 
equivariant algebraic $K$-theory in the unstable homotopy category.

Recall from ~\eqref{eqn:adjunc} that given $X \in \Sm^G_k$, there is an 
adjoint pair of functors $({\rm Fr}_X,{\rm Ev}_X)$ between ${\sS}_{\bullet}$ 
and $\sM^G_{k, \bullet}$.

\begin{lem}\label{lem:QA-Psh}
The functors $({\rm Fr}_X, {\rm Ev}_X)$ form a Quillen pair with respect to the 
schemewise projective, local projective, and motivic projective model 
structures on $\sM^G_{k, \bullet}$. The same holds for the 
various localizations of the injective model structure.
\end{lem}
\begin{proof}
Recall that this adjunction is given by the maps
\[
\theta: \Hom_{{\rm \sS}_{\bullet}}\left(K, \sS(X_{+}, \sX)\right)
\to \Hom_{\sM^G_{k, \bullet}} 
\left(K \wedge X_+, \sX \right)
\]
\[
\theta(f)(a \wedge x) = f_a(x) 
\]
and 
\[
\phi: \Hom_{\sM^G_{k, \bullet}} 
\left(K \wedge X_+, \sX \right) \to   
\Hom_{{\sS_{\bullet}}}\left(K, \sS(X_{+},\sX)\right)
\]
\[
\phi(g)(a) = \left(x \mapsto g(a \wedge x)\right).
\]
It is straightforward to check that the maps are inverses to each other.

To show that $({\rm Fr}_X, {\rm Ev}_X)$ is a Quillen pair, we shall note 
that ${\rm Fr}_X$ preserves cofibrations and trivial cofibrations with 
respect to all the model structures given in the lemma. 
First we reduce to the schemewise projective model structure on 
$\sM^G_{k,\bullet}$.
This follows because schemewise weak equivalences are the coarsest types of 
weak equivalences under consideration 
(see \cite[Proposition~3.3.3]{Hirsc}), 
and likewise for the projective cofibrations.

Suppose that $f : K \to L$ be a cofibration (which is same as a monomorphism)
of pointed simplicial sets.
If $f$ is a weak equivalence, then for any pointed simplicial set $M$, the map
$K \wedge M \to L \wedge M$ is also a weak equivalence. In particular,
the map $K \wedge \sS(U, X_{+}) \to L \wedge \sS(U, X_{+})$ is a weak 
equivalence for any $U \in \Sm^G_k$.
Equivalently, the map $(K \wedge X_{+}) (U) \to (L \wedge X_{+}) (U)$ is a weak
equivalence for every $U \in \Sm^G_k$. But this is same as saying that
the map $K \wedge X_+ \to L \wedge X_+$ is a schemewise weak equivalence.

We now show that $K \wedge X_+ \to L \wedge X_+$ is a projective cofibration.
We consider a diagram in $\sM^G_{k, \bullet}$
\begin{equation}\label{eqn:QA-Psh1}
\xymatrix@C.9pc{
K \wedge X_{+} \ar[r] \ar[d] & \sX \ar[d]^{p} \\
L \wedge X_{+} \ar@{.>}[ru] \ar[r] & \sY}
\end{equation}
where $p$ is a projective trivial fibration.
It follows from the definitions of the maps $\theta$ and $\psi$ above that
the assignments
\[
\Hom_{\sM^G_{k, \bullet}} \left(L \wedge X_{+}, \sX \right)
\to \Hom_{{\rm \sS_{\bullet}}} \left(L, \sX(X) \right) \to
\Hom_{\sM^G_{k, \bullet}} \left(L \wedge X_{+}, \sX \right)
\]
\[
h \mapsto \left(a \mapsto h(a, {\rm id}_X) \right) ; \ \ 
h' \mapsto \left( (a \wedge (U \xrightarrow{u} X)) \mapsto
h'(a) \circ u \right)
\]
give bijective correspondences of the sets. Thus giving a lifting 
in ~\eqref{eqn:QA-Psh1} is equivalent to giving a lifting in the 
parallel diagram of simplicial sets
\begin{equation}\label{eqn:QA-Psh2}
\xymatrix@C.9pc{
K \ar[r] \ar[d] & \sX(X) \ar[d]^{p} \\
L \ar@{.>}[ru] \ar[r] & \sY(X).}
\end{equation}
Since the fibrations and weak equivalences in the schemewise projective
model structure are objectwise, we see from our assumption that the right 
vertical arrow in ~\eqref{eqn:QA-Psh2} is a trivial fibration in 
$\sS_{\bullet}$.
Since $K \to L$ is assumed to be a cofibration, we get the desired lifting
using the model structure on simplicial sets. This completes the proof of the
lemma.
\end{proof}

\begin{prop}\label{prop:DQA-Psh}
Let $(\sX, x)$ be a fibrant pointed motivic $G$-space in the 
local injective model structure. Then for any pointed simplicial set $K$ and 
any $X \in \Sm^G_k$, the Quillen pair $({\rm Fr}_X, {\rm Ev}_X)$ of 
Lemma~\ref{lem:QA-Psh} gives a canonical isomorphism
\[
\Hom_{{\rm Ho}^G_{eN, \bullet}(k)} \left(K \wedge X_+, \sX \right)
\xrightarrow{\simeq} [K, \sX(X)].
\]
If $\sX$ is also $\A^1$-local, then there is a canonical isomorphism
\[
\Hom_{{\rm Ho}^G_{\A^1, \bullet}(k)} \left(K \wedge X_+, \sX \right)
\xrightarrow{\simeq} [K, \sX(X)].
\]
\end{prop}
\begin{proof}
Since the functor $K \mapsto K \wedge X_+$ preserves weak equivalence 
in all our model structures and since
$\sX$ is fibrant in the local injective model structure, we conclude from 
Lemma~\ref{lem:QA-Psh} that there are isomorphisms
\[
\begin{array}{lll}
\Hom_{{\rm Ho}^G_{eN, \bullet}(k)} \left(K \wedge X_+, \sX \right)
& \simeq & \Hom_{{\rm Ho}^G_{eN, \bullet}(k)} 
\left({\bf L}{\rm Fr}_X(K), \sX \right) \\
& \simeq &
\Hom_{\sS_{\bullet}} \left(K, {\bf R}{\rm Ev}_X (\sX) \right) \\
& \simeq & \Hom_{\sS_{\bullet}} \left(K, {\rm Ev}_X (\sX) \right) \\
& \simeq & \Hom_{\sS_{\bullet}} \left(K, \sX(X) \right).
\end{array}
\]
Since $\sX$ is fibrant in the local injective model structure, it 
is schemewise fibrant. In particular, $\sX(X)$ is a Kan complex and hence
the last term is same as $[K, \sX(X)]$.

If $\sX$ is also $\A^1$-local, then it is $\A^1$-fibrant by 
Lemma~\ref{lem:A1-compare}. We can now repeat the above
argument using Lemma~\ref{lem:QA-Psh}.
\end{proof}

\begin{cor}\label{cor:ELMEN}
A map $f: \sX \to \sY$ of $\A^1$-fibrant pointed motivic $G$-spaces is a 
schemewise weak equivalence if and only if the map
\begin{equation}\label{eqn:Iso-point}
\Hom_{{\rm Ho}^G_{\A^1, \bullet}(k)}(S^i_s \wedge X_{+}, \sX) \to
\Hom_{{\rm Ho}^G_{\A^1, \bullet}(k)}(S^i_s \wedge X_{+}, \sY) 
\end{equation}
is an isomorphism for all $X \in \Sm^G_k$ and all $i \ge 0$.
\end{cor}
\begin{proof}
We only need to show the `if' part.
Since $\sX$ and $\sY$ are $\A^1$-fibrant, it follows from 
Proposition~\ref{prop:DQA-Psh} that the terms on the left and the right in 
~\eqref{eqn:Iso-point} are $\pi_i(\sX(X))$ and $\pi_i(\sY(X))$, respectively.
This implies that $\sX \to \sY$ is a schemewise weak equivalence
(and hence motivic weak equivalence).
\end{proof}

\subsubsection{Homotopy groups}\label{subsubsection:EHGS}
For a motivic $G$-space $\sX$, let $\pi^{G, \A^1}_0(\sX)$ be the
$eN$-sheaf associated to the presheaf $U \mapsto [U, \sX]_{G, \A^1}$ on
$\Sm^G_k$. We shall say that $\sX$ is equivariantly $\A^1$-connected if
$\pi^{G, \A^1}_0(\sX)$ is constant.

For a pointed motivic $G$-space $(\sX, x)$, let 
$\pi^{G, \A^1}_i(\sX, x)$ be the $eN$-sheaf associated to the presheaf 
$U \mapsto [S^i_s\wedge U_{+}, (\sX, x)]_{G, \A^1_{\bullet}}$. 

It follows from Corollary~\ref{cor:ELMEN} that if $\sX \to \sF$ is an
$\A^1$-fibrant replacement, then $\pi^{G, \A^1}_i(\sX, x)$ is same as the sheaf
associated to the presheaf of homotopy groups of the simplicial presheaf $\sF$.
It follows that $\pi^{G, \A^1}_i(\sX, x)$ is a sheaf of groups for 
$i \ge 1$ and a sheaf of abelian groups for $i \ge 2$.
Using the functorial fibrant replacements and Corollary~\ref{cor:ELMEN}, 
we obtain the following result.

\begin{prop}\label{prop:whitehead}
A morphism $f: \sX \to \sY$ of equivariantly $\A^1$-connected motivic 
$G$-spaces is a motivic weak equivalence if and only if for any choice of
base point $x \in \sX$, the induced map
\[
\pi^{G, \A^1}_i(\sX, x) \to  \pi^{G, \A^1}_i(\sY, f(x))
\]
is an isomorphism for all $i \ge 1$.
\end{prop}

\section{Comparison with the Nisnevich site and base change}
\label{section:Site-Com}
In this section, we study the connection of our $eN$-site with
various other sites associated with group scheme actions.

\subsection{Comparison with the Nisnevich site}
\label{subsection:Sub-G}
Suppose that $H \subseteq G$ is a subgroup. We then have the canonical
restriction functor $r^G_H: \Sch^G_k \to \Sch^H_k$. This functor has 
a left adjoint $e^G_H: \Sch^H_k \to \Sch^G_k$ given by $e^G_H(X) = 
G \stackrel{H}{\times} X$.  In particular, $r^G_H$ commutes with limits.
Thus we get the map of $eN$-sites $\wh{r}^G_H: \Sch^H_{k/Nis} \to \Sch^G_{k/Nis}$
and ${\wh{e}}^G_H: \Sch^G_{k/Nis} \to \Sch^H_{k/Nis}$. These
are not continuous since they do not in general preserve $eN$-covers
(see Proposition~\ref{prop:site-N}). However, if $H$ is the trivial subgroup 
scheme, then $\wh{r}^G_H$ and ${\wh{e}}^G_H$ preserves covers.  
Since the underlying topologies are
sub-canonical (see Corollary~\ref{cor:sub-can}), 
Proposition~\ref{prop:site-N} implies that 
$\wh{r}^G: \Sch_{k/Nis} \to \Sch^G_{k/Nis}$
is a morphism of sites. In the smooth setting, 
we denote the analogous functor by 
\begin{equation}\label{eqn:Res-1}
{\rm res}: \Sm_{k/Nis} \to \Sm^G_{k/Nis}.
\end{equation} 

\begin{lem}\label{lem:Res-pullback}
The pullback functor ${\rm res}^*: \sM^G_k \to \sM_k$ preserves representable 
sheaves. It preserves local and motivic weak equivalences. 
\end{lem}
\begin{proof}
To show that ${\rm res}^*(X)(U) = \Hom_{\Sm_S}(U,X)$ for $X\in \Sm^G_S$ and $U\in \Sm_S$, 
notice that the term involving the pullback functor is ${\underset{\{U \to V| V \in \Sm^G_S\}}\colim} \ \Hom_{\Sm^G_S}(V, X)$. 
But the colimit is clearly same as the set $\Hom_{\Sm_S}(U,X)$.

As ${\rm res}$ is a morphism of sites, 
${\rm res}^*$ preserves local (with respect to the equivariant and ordinary Nisnevich topologies) 
weak equivalences by \cite[Proposition~2.1.47]{MV}.

Suppose now that $\sX$ is an $\A^1$-local object of $\sM_k$ and let $X \times \A^1 \to X$ be the projection map for some $X \in \Sm^G_S$.
Then $S_{\sM^G_S}(X, \sX)$ identifies with $\sX(r^G(X))$ and likewise for $X \times \A^1$. 
It follows that ${\rm res}_*(\sX)$ is $\A^1$-local in $\sM^G_S$. 
Combined with the adjunction it follows that ${\rm res}^*$ preserves motivic weak equivalences.   
\end{proof}

Using Lemma~\ref{lem:Res-pullback} and \cite[Proposition~2.3.17]{MV},
we get the following result.

\begin{prop}\label{prop:Res-Main}
The map ${\rm res}: \Sm_{k, Nis} \to \Sm^G_{k, Nis}$ is a
morphism of sites such that ${\rm res}^*$ preserves local
and motivic weak equivalences. Furthermore, there is an adjoint pair of
functors
\[
\xymatrix{
{\bf L}{\rm res}^*: {\rm Ho}^G_{\A^1}(k) 
\ar@<+.7ex>[r] &
\ar@<+.7ex>[l]
{\rm Ho}_{\A^1}(k) : {\bf R}{\rm res}_*.
}
\]
\end{prop}

We note the following immediate corollary in connection with representability 
of equivariant $K$-theory, see \S~\ref{section:Nis-Desc-rep}.

\begin{cor}\label{cor:Res-Main-I}
Let $f : X \to Y$ be a $G$-equivariant map of smooth $G$-schemes. Suppose
that $f$ is a motivic weak equivalence in $\sM^G_k$. Then the induced 
map $K_*(Y) \xrightarrow{f^*} K_*(X)$ is an isomorphism of ordinary $K$-theory.
\end{cor}

Recall from ~\eqref{eqn:trivial-embedding} that there is a full and faithful
embedding $\Sm_k \to \Sm^G_k$, which takes a scheme $X$ to itself
with the trivial $G$-action. This functor commutes with fiber product
and takes a Nisnevich cover to an $eN$-cover. 
Corollary~\ref{cor:sub-can} and Proposition~\ref{prop:site-N} imply that 
there is an induced morphism of sites $\iota: \Sm^G_{k/Nis} \to \Sm_{k/Nis}$.
Note that $\iota^*$ is identity and $\iota_*$ takes any $G$-scheme $X$
to the fixed point subscheme $X^G$. Recall that $X^G$ is smooth. 
We get the following result.

\begin{prop}\label{prop:TR-embed-HO}
The morphism of sites $\iota: \Sm^G_{k, Nis} \to \Sm_{k, Nis}$ 
induces a pair of adjoint
functors
\[
\xymatrix{
\iota^*: {\rm Ho}_{\A^1}(k)
\ar@<+.7ex>[r] &
\ar@<+.7ex>[l]
{\rm Ho}^G_{\A^1}(k) : {\bf R}{\iota}_*.
}
\]
The functor $\iota^*$ is a full and faithful embedding of the
motivic homotopy category of smooth schemes into the equivariant motivic 
homotopy category.
\end{prop}

\vskip .4cm

\subsection{Change of base field}\label{subsection:Base-change}
Suppose now that $k \inj k'$ is an extension of fields and
set $G' = G {\underset{\Spec(k)}\times} \ \Spec(k')$.
Notice that $G'$ is identified with $G$ if the latter is a finite constant
group scheme over $k$.
The base change functor $f^{-1}:\Sm^G_k \to \Sm^{G'}_{k'}$ is defined by
$X \mapsto X {\underset{\Spec(k)}\times} \ \Spec(k')$.
It is clear that $f^{-1}$ preserves distinguished $eN$-squares.
Thus Corollary~\ref{cor:Sheaf-Nis} shows that the site map 
$f: \Sch^{G'}_{{k'}/Nis} \to \Sch^G_{k/ Nis}$ is continuous.
Since $f^{-1}$ clearly commutes with fiber products, it follows from
Corollary~\ref{cor:sub-can} and Proposition~\ref{prop:site-N} that
$f$ is a morphism of sites.

\begin{prop}\label{prop:Base-Change-Main}
Given an extension of fields $k \inj k'$, the base change 
functor $f^{-1}$ induces a morphism of sites
$f: \Sm^{G'}_{{k'}/Nis} \to \Sm^{G}_{k/ Nis}$. This yields an adjoint pair
of functors  
\[
\xymatrix{
{\bf L}{f^*}: {\rm Ho}^G_{\A^1}(k) 
\ar@<+.7ex>[r] &
\ar@<+.7ex>[l]
{\rm Ho}^{G'}_{\A^1}(k') : {\bf R}{f_*}.
}
\]

If $k \inj k'$ is a finite separable extension, then $f^*$ has a left adjoint
$f_{\#}: \sM^{G'}_{k'} \to \sM^G_k$ which takes any $U \in \Sm^{G'}_{k'}$ to 
itself, viewed as a $G$-scheme over $k$. 
This functor preserves motivic weak equivalences and
$f^*$ preserves $\A^1$-local motivic $G$-spaces. There is an 
adjoint pair of functors  
\[
\xymatrix{
{\bf L}{f_{\#}}: {\rm Ho}^{G'}_{\A^1}(k') 
\ar@<+.7ex>[r] &
\ar@<+.7ex>[l]
{\rm Ho}^G_{\A^1}(k) : {\bf L}{f^*}.
}
\]
\end{prop}

\vskip .3cm

\section{Local $eN$-linearization of $G$-schemes}
\label{section:ENL-G-sch}
The homotopy purity theorem (see \cite[Theorem~3.2.23]{MV}) is one of the 
most important tools in $\A^1$-homotopy theory, 
e.g., in the construction of Gysin long exact sequences and for Poincar{\'e} 
duality in its most concise form.
Our goal in this and the following section is to 
establish the purity theorem for $G$-schemes when $G$ is a finite cyclic group
of prime order. 
This theorem turns out to have many applications in 
the equivariant motivic stable homotopy category. 
As part of proving the purity theorem, we first establish a local equivariant 
linearization of smooth $G$-schemes in the Zariski topology.

\subsection{$eN$-linearization near a fixed point}\label{subsection:G-lin}
We shall assume throughout this section that $G$ is a finite constant group 
scheme over $k$ of order
prime to the characteristic of $k$. This is mainly to ensure
that $G$ is linearly reductive. A (finite) $G$-module will mean a 
(finite-dimensional) rational representation of $G$. 
We begin with the following elementary result about $G$-modules.

\begin{lem}\label{lem:G-mod}
Consider a commutative diagram of $G$-modules
\begin{equation}\label{eqn:G-mod-1}    
\xymatrix@C1pc{
0 \ar[r] & M_1 \ar[r] \ar[d]_{u_1} & M \ar[r]^{v} \ar[d]_{u} & M_2 \ar[r] 
\ar[d]^{u_2} & 0 \\
0 \ar[r] & N_1 \ar[r] & N \ar[r]_{v'} & N_2 \ar[r] & 0}
\end{equation}
in which the rows are exact and the vertical maps are surjective.
Assume that $N$ is a finite $G$-module.
Then there exists a finite $G$-submodule $M' \subseteq M$ and commutative
diagram of finite $G$-modules
\begin{equation}\label{eqn:G-mod-2}    
\xymatrix@C1pc{
0 \ar[r] & M'_1 \ar[r] \ar[d]_{u'_1} & M' \ar[r] \ar[d]_{u'} & M'_2 \ar[r] 
\ar[d]^{u'_2} & 0 \\
0 \ar[r] & N_1 \ar[r] & N \ar[r] & N_2 \ar[r] & 0}
\end{equation}
with exact rows such that the vertical maps are the
restriction of the vertical maps of ~\eqref{eqn:G-mod-1} to $G$-submodules.
Moreover, they are all isomorphisms.
\end{lem}
\begin{proof}
This is an application of the fact that $G$ is linearly reductive.
We give a sketch of the proof.
Since $N$ is a finite $G$-module, so are $N_1$ and $N_2$. Hence, we can first
find a finite-dimensional $k$-linear subspace $V \subseteq M$ such that
$u(V) = N$. We can then find inclusions of linear subspaces
$V \subseteq V' \subseteq M$ such that $V'$ is a finite $G$-submodule
and $u(V') = N$. Set $L = {\rm ker}(V' \surj N)$.

Since $G$ is linearly reductive, its representation theory tells us that 
there is a decomposition $N = N_1 \oplus N'_2$ of finite $G$-modules
such that $N'_2$ is mapped isomorphically onto $N_2$.  
Similarly, there is a direct sum decomposition of finite $G$-modules
$V' = L \oplus N'$ such that $N'$ is mapped isomorphically onto $N$ via $u$.

We now set $M' = N', \ M'_2 = v\left(u^{-1}(N'_2) \cap M' \right)$ and
$M'_1 = {\rm Ker}(M' \surj M'_2)$. It is easy to check that we get a
diagram as required in ~\eqref{eqn:G-mod-2}.
\end{proof}

Given a smooth scheme $X$ and a closed point $x \in X$, let $T_xX$ denote the
tangent space of $X$ at $x$. Notice that if $X \in \Sm^G_k$ and if
$x \in X^G$, then $G$ naturally acts $k(x)$-linearly on $T_xX$.
For an affine scheme $X$, its ring of regular functions will be denoted by
$k[X]$.

\begin{lem}\label{lem:Linearization}
Let $X \in \Sm^G_k$ be an affine scheme and let $Z \subsetneq X$  
be a smooth $G$-invariant closed subscheme. Let $x \in Z$ be a $k$-rational 
point such that $x \in X^G$. Then there is a $G$-invariant affine neighborhood
$U \subseteq X$ of $x$ and a $G$-equivariant {\'e}tale map
$f:U \to T_xX$ such that $f^{-1}(T_xZ) = Z \cap U$.
\end{lem}
\begin{proof}
Let $\fm_X \subsetneq k[X]$ denote the maximal ideal defining the closed point 
$x$. Since $x \in X^G$, we see that $\fm_X$ (and all its powers) acquires 
natural $G$-action coming from the $G$-action on $k[X]$ and the surjection
$u: \fm_X \surj {\fm_X}/{\fm^2_X} = (T_xX)^*$ is an $H$-equivariant $k$-linear
map. Let $I \subsetneq k[X]$ denote the ideal defining the closed subscheme
$Z$. Then $I$ is also $G$-invariant under $G$-action on $k[X]$.
Thus we get a commutative diagram of $G$-modules and $G$-linear maps:
\begin{equation}\label{eqn:Linearization-1}    
\xymatrix@C1pc{
0 \ar[r] & I \ar[r] \ar[d] & \fm_X \ar[r] \ar[d] & \fm_Z \ar[r] \ar[d] & 0 \\
0 \ar[r] & \frac{I}{I \cap \fm^2_X} \ar[r] & \frac{\fm_X}{\fm^2_X} 
\ar[r] & \frac{\fm_Z}{\fm^2_Z} \ar[r] & 0}
\end{equation}
in which the rows are exact, the vertical maps are surjective and
the bottom row consists of finite $G$-modules.

We can now apply Lemma~\ref{lem:G-mod} to get a commutative diagram
of exact sequences of finite $G$-modules:
\begin{equation}\label{eqn:Linearization-2}    
\xymatrix@C1pc{
0 \ar[r] & M_{(X,Z)} \ar[d]_{u_{(X,Z)}} \ar[r] & M_X \ar[r] \ar[d]_{u_X} & 
M_Z \ar[r] \ar[d]^{u_Z} & 0 \\
0 \ar[r] & (N_xZ)^* \ar[r] & (T_xX)^* \ar[r] & (T_xZ)^* \ar[r] & 0}
\end{equation}
such that the vertical maps are all isomorphisms,
Here $N_xZ$ denotes the normal space of $Z \inj X$ at $x$. 
Moreover, the top row is a sequence of $G$-submodules of the top row
of ~\eqref{eqn:Linearization-1}.
Notice also that as part of the proof of Lemma~\ref{lem:G-mod}, we have
shown that there is a $k$-basis of $M_X$ which maps onto the $k$-bases of
$M_Z$ as well as $(T_xX)^*$. 

Using these bases, we can now construct a commutative diagram of
exact sequences of finite $G$-modules:
\begin{equation}\label{eqn:Linearization-3}    
\xymatrix@C1pc{
0 \ar[r] & (N_xZ)^* \ar[r] \ar[d]_{u^{-1}_{(X,Z)}}^{\simeq} & (T_xX)^* \ar[r] 
\ar[d]_{u^{-1}_X}^{\simeq} & (T_xZ)^* \ar[r] \ar[d]^{u^{-1}_Z}_{\simeq} & 0 \\
0 \ar[r] & M_{(X,Z)}  \ar[r] & M_X \ar[r] & M_Z \ar[r] & 0}
\end{equation}
such that the vertical maps are isomorphisms. 

The maps $u^{-1}_X$ and $u^{-1}_Z$ induce the corresponding $G$-equivariant
maps of the associated symmetric algebras over $k$ (recall that $x \in X(k)$)
and composing these maps
of symmetric algebras with inclusions ${\rm Sym}^*(M_X) \inj k[X]$ and
${\rm Sym}^*(M_Z) \inj k[Z]$, we get a commutative diagram of
$G$-equivariant morphisms

\begin{equation}\label{eqn:Linearization-4}    
\xymatrix@C1pc{
{\rm Sym}^*_k((T_xX)^*) \ar@{->>}[r] 
\ar[d]_{u^{-1}_X} & {\rm Sym}^*_k((T_xZ)^*) \ar[d]^{u^{-1}_Z} \\
k[X]  \ar@{->>}[r] & k[Z].}
\end{equation}
To check that the kernel of the top row maps onto the ideal $I$ locally at the
closed point $x$, we just have to observe from ~\eqref{eqn:Linearization-1}
that $(N_xZ)^*$ is nothing but $I/{(I \cap \fm^2_X)}$ and it maps to the ideal
of $Z$ near $x$ via $u^{-1}_X$. 

It is easy to check from the local criterion of flatness that $u^{-1}_X$ is 
flat near $x$. Furthermore, $u^{-1}_X$ clearly induces an isomorphism of the
tangent spaces at $\fm_X$ and $u^{-1}_X(\fm_X)$. If we set $f$ to be the 
morphism $f: X \to T_xX$ defined by $u^{-1}_X$, we see that $f$ is an 
$G$-equivariant morphism which is {\'e}tale at $x$ and $f^{-1}(T_xZ) = Z$
near $x$. We conclude that there is an affine  neighborhood 
$U' \subseteq X$ of $x$
such that the restriction $f_{U'}$ on $U'$ is {\'e}tale and
$f^{-1}_{U'}(T_xZ) = Z \cap U'$. Finally, using the fact that $x \in X^G$, we
set $U = {\underset{g \in G}\cap} \ gU'$ and conclude that $U \subseteq X$
is a $G$-invariant affine neighborhood of $x$ and there is a $G$-equivariant
{\'e}tale map $f: U \to T_xX$ such that $f^{-1}(T_xZ) = Z \cap U$.
The proof of the lemma is now complete.
\end{proof}

\begin{prop}\label{prop:Linearization-Main-III}
Let $G$ be a finite cyclic group of prime order $p$ which is different from
the characteristic of $k$. Let $Z \inj X$ be a closed immersion
in $\Sm^G_k$ and $x \in Z$ a $k$-rational point. Then, there is a 
$G$-invariant affine neighborhood $U$ of $x$, a $G$-representation $V$ 
with $G$-submodule $Z_V$ and a $G$-equivariant {\'e}tale map
$f: U \to V$ such that $f^{-1}(Z_V) = Z \cap U$.
\end{prop}
\begin{proof}
Let $X^G$ denote the closed subscheme of fixed points for the $G$-action on
$X$. We first assume that $x \notin X^G$. Since $X \setminus X^G$ is 
$G$-invariant, we can assume that $X^G = \0$.
Since $G$ is a cyclic group of prime order, it acts freely on $X$. In 
particular, the quotient map $\pi: X \to X/G$ is finite {\'e}tale of degree 
$p$. Set $X' = X/G$ and $Z' = Z/G$. Then we see that $p$ is a ($G$-equivariant)
finite {\'e}tale map with $\pi^{-1}(Z') =   Z$. 

Since $(X', Z')$ is a closed immersion
of smooth schemes over $k$, we know that there is an affine neighborhood
$U'$ of $x' = \pi(x)$ in $X'$ and an {\'e}tale map
$f': U \to \A^d_k$ such that $f'^{-1}(\A^c_k \times \{0\}) = Z'$
for some $1 \le c \le d$. Setting $f = f' \circ \pi$ and $U = \pi^{-1}(U')$,
we conclude that $U$ is a $G$-invariant affine neighborhood of $x$.
Moreover, there is a $G$-equivariant {\'e}tale map $f: U \to \A^d_k$ (with 
respect to the trivial action on $\A^d_k$) such that
$f^{-1}(\A^c_k \times \{0\}) = Z \cap U$.

We next suppose that $x \in X^G$. Let $U'$ be an affine neighborhood of
$x$ in $X$. Since $G_x = G$, we see that $S_x = G$. In particular, 
$U = {\underset{g \in G}\cap} gU'$ is a $G$-invariant affine neighborhood of
$x$. We can thus
assume that $X$ is affine. It follows now from Lemma~\ref{lem:Linearization}
that there is a $G$-invariant affine neighborhood $U$ of $x$ in $X$ and
a $G$-equivariant  {\'e}tale map $f: U \to T_xX$ such that $f^{-1}(T_xZ) =
Z \cap U$. Moreover, as $p \neq {\rm char}(k)$, there is a $G$-equivariant
decomposition $T_xX = T_xZ \times N_xX$.
\end{proof}

\begin{defn}\label{defn:EN-Linearization}
Given a closed immersion $Z \inj X$ in $\Sm^G_k$, 
an $eN$-linearization of the pair $(X, Z)$ is
a pair $(p,q )$ of maps in $\Sm^G_k$ given by
\begin{equation}\label{eqn:ENL-1}
(X, Z) \xleftarrow{p} (U, Z) \xrightarrow{q} (N_{Z/X}, Z)
\end{equation}
such that $p$ and $q$ are both distinguished $eN$-neighborhoods.
We shall say that $(X,Z)$ admits an $eN$-linearization if
the pair $(p, q)$ as in ~\eqref{eqn:ENL-1} exists.
\end{defn}   

\begin{prop}\label{prop:Linearization-Main-II}
Let $G$ be a finite cyclic group of prime order $p$ which is different from
the characteristic of $k$. Let $Z \inj X$ be a closed immersion
in $\Sm^G_k$ and $x \in Z$ a $k$-rational point. Then, there is a 
$G$-invariant affine neighborhood $U$ of $x$ such that the
pair $(W, w^{-1}{(Z \cap U)})$ admits an $eN$-linearization for any 
$G$-equivariant {\'e}tale map $w: W \to U$.
\end{prop}
\begin{proof}
Given any map $W \to X$, we set $Z_W= Z {\underset{X}\times} W$.
We choose a $G$-invariant affine neighborhood $U$ of $x$, 
and a $G$-equivariant {\'e}tale map $f: U \to V$ as in
Proposition~\ref{prop:Linearization-Main-III}.
Let $f_Z:Z_U \to Z_V$ denote the restriction of $f$ to $Z$.

Since $p \neq {\rm char}(k)$, there is a $G$-equivariant decomposition
$V = Z_V \times N_{Z/V}$. Let $j: N_{Z/V} \inj V$ be the inclusion map
and let $f': Z_U \times N_{Z/V} \to V$ be the $G$-equivariant map
$f_Z \times j$. We now consider a commutative diagram in $\Sm^G_k$: 
\begin{equation}\label{eqn:Lin-Main-II-1}
\xymatrix@C1pc{
Z_U \ar[dr]^{\wt{i}} \ar@/_1.5pc/[ddr]_{i} 
\ar@/^1pc/[drr]^{i'} & & \\
& \wt{U} \ar[r]^<<{q} \ar[d]_{p} & Z_U \times N_{Z/V} 
\ar[d]^{f'} \\
& U \ar[r]_{f} & V}
\end{equation}
in which $\wt{U}$ is defined so that the square is Cartesian and $i'$ is the
zero sectioninclusion. Notice that $\wt{U}$ is smooth since 
($f$ and hence) $q$ is {\'e}tale.

It is easy to check that $(f' \circ q)^{-1}(Z_V)$ 
is the $G$-invariant closed subscheme $Z_U {\underset{V}\times} Z_U$. Since
$Z_U \to Z_V$ (obtained by the restriction of $f$) is {\'e}tale by 
Proposition~\ref{prop:Linearization-Main-III}, 
we see that this closed subscheme
is a disjoint union of diagonal $\Delta_{Z_U}: Z_U \inj Z_U 
{\underset{V}\times} Z_U$ and a $G$-invariant closed subscheme $Y$. 
In particular, $Y$ is a
$G$-invariant closed subscheme of $\wt{U}$. Setting $\wh{U} = 
\wt{U} \setminus Y$,
we get $G$-equivariant {\'e}tale maps 
$p: \wh{U} \to U$ and 
$q: \wh{U} \to V$
and one checks from the construction that 
$p^{-1}(Z_U) = q^{-1}(Z_V) = 
\wt{i}(Z_U)$.

If $w: W \to U$ is a $G$-equivariant {\'e}tale
map, then we have $N_{{Z_W}/W} \simeq Z_W {\underset{Z_U}\times} N_{{Z_U}/U}$.
Let $w_Z: N_{{Z_W}/W} \to N_{{Z_V}/V} \simeq V$ denote the
projection map. This yields an analogous  commutative diagram:
\begin{equation}\label{eqn:Lin-Main-II-2}
\xymatrix@C1.5pc{
Z_W \ar[dr]^{\wt{i}} \ar@/_1.5pc/[ddr]_{i} 
\ar@/^1pc/[drr]^{i'} & & \\
& \wt{W} \ar[r]^<<<{q} \ar[d]_{p} & N_{{Z_W}/W} \ar[d]^{f'\circ w_Z} \\
& W \ar[r]_{f \circ w} & V}
\end{equation}
where the lower square is Cartesian.
We now repeat the above construction for $f:U \to V$
verbatim to get $G$-equivariant {\'e}tale maps 
$p: \wh{W} \to W$ and 
$q: \wh{W} \to N_{{Z_W}/W}$
and one checks from the construction that 
$p^{-1}(Z_W) =  q^{-1}(Z_W) = \wt{i}(Z_W)$, where
$Z_W \subset N_{{Z_W}/W}$ is the zero-section.
\end{proof}

\vskip .3cm

\section{The equivariant homotopy purity and blow-up theorems}
\label{section:EPT}

The {\sl equivariant Thom space} of a $G$-equivariant vector bundle $V \to X$ in $\Sm^G_k$ 
is the pointed motivic $G$-space $V/{(V \setminus X)}$,
where $X \inj V$ is the zero section. 
We prove the following purity theorem for normal bundles.

\begin{thm}\label{thm:Purity}
Let $k$ be an algebraically closed field and let
$G$ be a finite cyclic group of prime order $p$ which is different from
the characteristic of $k$.
Let $Z \inj X$ be a closed immersion in $\Sm^G_k$.
Then there is a canonical isomorphism in 
${\rm Ho}^G_{\A^1, \bullet}(k)$
of pointed motivic $G$-spaces
\[
X/{(X \setminus Z)} \simeq {\rm Th}(N_{Z/X}).
\]
\end{thm}

The proof combines our results on $eN$-linearizations with the ideas from the non-equivariant set-up in \cite{MV}.

\subsection{Purity for vector bundles}
Let $k$ be any field and $G$ any finite group.
Let $Z \inj X$ be a closed immersion in $\Sm^G_k$. 
Consider $\A^1_k$ with trivial
$G$-action and let $B(X,Z)$ denote the blow-up of $X \times \A^1$ along the
$G$-invariant closed subscheme $Z \times \{0\}$. 
It is straightforward to check that $B(X,Z) \in \Sm^G_k$ and the blow-up map
$f: B(X,Z) \to X \times \A^1$ is $G$-equivariant. 
Furthermore, it is standard that there are inclusions of closed pairs
in $\Sm^G_k$
\begin{equation}\label{eqn:Blow-up}
(X, Z) \xrightarrow{i_1} (B(X, Z), Z \times \A^1) \xleftarrow{i_0}
(\P(N_{Z/X} \times \A^1), Z),
\end{equation}
where the inclusion in the last pair is the composition
$Z \inj N_{Z/X} = \P(N_{Z/X} \times \A^1) \setminus \P(N_{Z/X})$.
Using the isomorphism ${\rm Th}(N_{Z/X}) \simeq \frac{\P(N_{Z/X} \times \A^1)}
{\P(N_{Z/X} \times \A^1) \setminus Z}$ (see \cite[Proposition~3.2.17]{MV}),
we get the monomorphisms of pointed motivic $G$-spaces
\begin{equation}\label{eqn:Blow-up-1}
\alpha_{X,Z} : \frac{X}{X \setminus Z} \to \frac{B(X, Z)}{B(X, Z) \setminus
(Z \times \A^1)} ;
\end{equation}
\[
\beta_{X, Z} : {\rm Th}(N_{Z/X}) \to
\frac{B(X, Z)}{B(X, Z) \setminus (Z \times \A^1)}.
\]

\begin{lem}\label{lem:Thom-VB}
Let $p: V \to Z$ be a $G$-equivariant vector bundle in $\Sm^G_k$ and let
$i: Z \inj V$ be the zero section. Then the maps
$\alpha_{V,Z}$ and $\beta_{V,Z}$ are motivic weak equivalences.
\end{lem}
\begin{proof}
We first recall that there is a natural map 
$\lambda_Z: B(V,Z) \to \P(V \times \A^1)$ in 
$\Sm^G_k$ which is the relative line bundle $\sO(1)$.
Moreover, one has $\lambda_Z^{-1}(\P(V \times \A^1) \setminus Z) = 
B(V,Z) \setminus (Z \times \A^1)$,
where $Z   \inj \P(V \times \A^1)$ is the inclusion
$Z \stackrel{i} \inj V = \P(V \times \A^1) \setminus \P(V)$.
In particular, these maps are motivic weak equivalences by
Proposition~\ref{prop:Hom-inv}. We conclude that the map
\[
q: \frac{B(V,Z)}{B(V,Z) \setminus (Z \times \A^1)} \to
\frac{\P(V \times \A^1)}{\P(V \times \A^1) \setminus Z}
\]
is a motivic weak equivalence.
On the other hand, the composite $q \circ \alpha_{V,Z}$ is a canonical
isomorphism of pointed motivic $G$-spaces (see \cite[Proposition~3.2.17]{MV}). 
We conclude that $\alpha_{V, Z}$ is a motivic weak equivalence.

On the other hand, the composition of the projection $\lambda_Z$ with
the inclusion $V \xrightarrow{i_0} B(V,Z)$ is the canonical open inclusion
$V \inj \P(V \times \A^1)$. Since
\begin{equation}\label{eqn:Thom-VB-1}
\xymatrix@C1pc{
V \setminus Z \ar[r] \ar[d] & V \ar[d] \\
\P(V \times \A^1) \setminus Z \ar[r] & \P(V \times \A^1)}
\end{equation}
is a distinguished $eN$-square, it follows that the composition
\[
\frac{V}{V \setminus Z} \xrightarrow{\beta_{V,Z}} 
\frac{B(V,Z)}{B(V,Z) \setminus (Z \times \A^1)} \xrightarrow{q}
\frac{\P(V \times \A^1)}{\P(V \times \A^1) \setminus Z}
\]
is a local weak equivalence. Since $q$ is a motivic weak equivalence,
we conclude that $\beta_{V,Z}$ is a motivic weak equivalence.
\end{proof}

\vskip .4cm

\subsection{Purity in general}
Let $(X,Z)$ be a closed pair in $\Sm^G_k$ as in Theorem~\ref{thm:Purity}.
It follows from Proposition~\ref{prop:Linearization-Main-II} that
for every $x \in Z$, there exists a $G$-invariant affine neighborhood 
$U$ of $x$ such that the pair $(U, Z\cap U)$ admits an $eN$-linearization.
Since $X$ is noetherian, there exists a finite set $\{U_1, \cdots , U_r\}$
of $G$-invariant affine open subsets of $X$ such that $X =
\stackrel{r}{\underset{i =1}\cup} U_i$ and each pair $(U_i,   Z \cap U_i)$
admits an $eN$-linearization.  

Set $U = \stackrel{r}{\underset{i =1}\coprod} U_i$ and 
$Z_U = \stackrel{r}{\underset{i =1}\coprod} (Z \cap U_i)$. 
Then $(U, Z_U)$ is a pair
of objects in the category $\Shv^G_{\Sm_k}$ and there is a canonical map
of sheaves $u: (U, Z_U) \to (X,Z)$. 

Let $\sU$ (resp.~$\sZ_{\sU}$)
denote the simplicial sheaf on $\Sm^G_{{\rm Nis}/k}$ whose term at 
level $n$ is the $(n+1)$-fold product $U {\underset{X}\times} \cdots
{\underset{X}\times} U$ (resp. $Z_U {\underset{Z}\times} \cdots
{\underset{Z}\times} Z_U$). 
This yields a pair of motivic $G$-spaces $(\sU, \sZ_{\sU})$ and a map
of pairs of motivic $G$-spaces $f: (\sU, \sZ_{\sU}) \to (X,Z)$.
Setting $U^n_X$ to be the $(n+1)$-fold product $U {\underset{X}\times} \cdots
{\underset{X}\times} U$, we see that $U^n_X$ is the coproduct of smooth
$G$-schemes each of which is a fiber product (over $X$) of $n+1$ components of 
$U$. Set $\sU \setminus \sZ_{\sU} = u^{-1}(X \setminus Z)$.

Let $\sB$ denote the motivic $G$-space obtained by applying the $B(X,Z)$
construction levelwise to the inclusion $\sZ_{\sU} \inj \sU$ 
(see \cite[p.~117]{MV}).
Observe here that this inclusion is the coproduct of closed embeddings
of smooth $G$-schemes at each level. Moreover, we have $B(X \setminus Z, \0) 
\simeq (X \setminus Z) \times \A^1$ and ${\rm Th}(N_{{\0}/{(X \setminus Z)}}) =
\Spec(k)$ (as a pointed motivic $G$-space).
Let ${\rm Th}(N_{{\sZ_{\sU}}/{\sU}})$ denote the motivic $G$-space which is 
obtained by applying the levelwise Thom space construction for the
inclusion $\sZ_{\sU} \inj N_{{\sZ_{\sU}}/{\sU}}$. This makes sense because
$\sZ_{\sU} \inj N_{{\sZ_{\sU}}/{\sU}}$ is the coproduct of 0-section embeddings
of equivariant vector bundles over smooth $G$-schemes at each level. 
We obtain a commutative diagram of pointed motivic $G$-spaces
\begin{equation}\label{eqn:P-2}
\xymatrix@C1pc{
\frac{\sU}{\sU \setminus \sZ_{\sU}} \ar[r] \ar[d] &
\frac{\sB}{\sB \setminus (\sZ_{\sU} \times \A^1)} \ar[d] &
{\rm Th}(N_{{\sZ_{\sU}}/{\sU}}) \ar[d] \ar[l] \\
\frac{X}{X \setminus Z} \ar[r] &
\frac{B(X, Z)}{B(X, Z) \setminus (Z \times \A^1)}  &
{\rm Th}(N_{Z/X}). \ar[l]}  \\
\end{equation}

\begin{lem}\label{lem:Purity-we}
The vertical arrows in ~\eqref{eqn:P-2} are local weak equivalences in the
$eN$-topology.
\end{lem}
\begin{proof}
It suffices to show that the left vertical arrow is a local weak equivalence
as the same argument shows this weak equivalence for the other two vertical 
arrows.
We first claim that the map of sheaves $U \to X$ is an epimorphism in the
$eN$-topology.
Using Proposition~\ref{prop:Point-Cons}, it suffices to show that
for any $Y \in \Sch^G_k$ and a point $y \in Y$, the map
$U(Y^h_{Gy}) \to X(Y^h_{Gy})$ is surjective. So let $v: Y^h_{Gy} \to X$ be
a $G$-equivariant morphism where $Y^h_{Gy}$ is the henselization of a 
$G$-scheme $Y$ along the $G$-orbit $Gy$. 
By our construction of $U$, there is
a component $U_x$ of the scheme $U$ such that $u: (U_x, Gx) \to (X, Gx)$ is
an affine (Zariski) neighborhood of $Gx$. 
In particular, the map $(U_x)^h_{Gx} \xrightarrow{u}
X^h_{Gx}$ is a $G$-equivariant isomorphism of semi-local $G$-schemes.

Now, the map $v$ induces a $G$-equivariant map $Y^h_{Gy} \to X^h_{Gx}$ which 
takes $Gy$ onto $Gx$. Since $u$ is a $G$-equivariant isomorphism, we
immediately get a $G$-equivariant morphism $w: Y^h_{Gy} \to (U_x)^h_{Gx}$ such 
that $v = u \circ w$. Composing $w$ with the canonical maps
$(U_x)^h_{Gx} \to U_x \inj U$, we get a map $Y^h_{Gx} \to U$ which factors
$v$. This proves the claim.   

Since the $eN$-topology on $\Sm^G_k$ admits a conservative family of points
by Proposition~\ref{prop:Point-Cons}, we can use the above claim and
\cite[Lemma~2.1.15]{MV} to conclude that the map $\sU \to X$ is a local
weak equivalence. For the same reason, the map $\sU \setminus \sZ_{\sU} \to
X \setminus Z$ is a local weak equivalence. We conclude from
\cite[Lemma~2.2.11]{MV} that also $\frac{\sU}{\sU \setminus \sZ_{\sU}} \to
\frac{X}{X \setminus Z}$ is a local weak equivalence.
\end{proof}

\vskip .3cm

{\sl Proof of Theorem~\ref{thm:Purity}:}
It suffices to show that the maps $\alpha_{X,Z}$ and $\beta_{X,Z}$ are
motivic weak equivalences. By Lemma~\ref{lem:Purity-we}, this is equivalent 
to showing that the top horizontal maps in ~\eqref{eqn:P-2} are motivic weak
equivalences. By \cite[Proposition~2.2.14]{MV}, it suffices to show that
the top horizontal maps of simplicial sheaves in ~\eqref{eqn:P-2} are motivic 
weak equivalences at each level $n \ge 0$. 

The top horizontal maps in ~\eqref{eqn:P-2}
are isomorphisms for $X \setminus Z$ and all its
$G$-invariant open subsets. Thus we are left with showing 
that the top horizontal maps in ~\eqref{eqn:P-2} are motivic weak
equivalences for closed pairs of the form $(U_x, Z_{U_x})$ and 
$(W, \psi^{-1}(Z_{U_x}))$, 
where $\psi: W \to U_x$ is a $G$-equivariant {\'e}tale map. 
By Proposition~\ref{prop:Linearization-Main-II}, we are reduced to
proving the theorem under the assumption that the closed pair $(X,Z)$ admits an
$eN$-linearization. 

So let $(X,Z) \xleftarrow{p} (U, Z) \xrightarrow{q} (N_{Z/X}, Z)$
be an $eN$-linearization of $(X,Z)$ and consider the commutative diagram
\begin{equation}\label{eqn:P-1}
\xymatrix@C1pc{
\frac{U}{U \setminus Z_U} \ar[r] \ar[d] &
\frac{B(U, Z_U)}{B(U, Z_U) \setminus (Z_U \times \A^1)} \ar[d] &
{\rm Th}(N_{{Z_U}/U}) \ar[d] \ar[l] \\
\frac{X}{X \setminus Z} \ar[r] &
\frac{B(X, Z)}{B(X, Z) \setminus (Z \times \A^1)}  &
{\rm Th}(N_{{Z}/X}). \ar[l]}  \\
\end{equation}
The vertical arrows in ~\eqref{eqn:P-1} are local weak equivalences by our 
definition of local weak equivalence in the $eN$-topology. Hence, the top
horizontal maps are motivic weak equivalences if and only if so are
the bottom horizontal maps.

If we apply this argument for $(N_{Z/X}, Z)$ in place of $(X,Z)$, it follows
from the local weak equivalence $\frac{U}{U \setminus Z_U} \xrightarrow{\simeq}
{\rm Th}(N_{{Z}/X})$ and Lemma~\ref{lem:Thom-VB} that
the top horizontal maps in ~\eqref{eqn:P-1} are motivic weak equivalences.
We conclude that the maps $\alpha_{X,Z}$ and $\beta_{X,Z}$ 
are motivic weak equivalences. 
$\hfill \square$

\vskip .3cm

Using the same line of proof as for Theorem~\ref{thm:Purity} verbatim, we
obtain the following result for equivariant blow-ups.

\begin{thm}\label{thm:Blow-up-square}
Let $k$ be an algebraically closed field and let
$G$ be a finite cyclic group of prime order $p$ which is different from
the characteristic of $k$.
Let $Z \inj X$ be a closed immersion in $\Sm^G_k$
with complement $U = X \setminus Z$.
Let $p: X' \to X$ denote the blow-up of $X$ along $Z$. Then the square
\begin{equation}\label{eqn:BUS-1}
\xymatrix@C1pc{
p^{-1}(Z) \ar[r] \ar[d] & {X'}/U \ar[d] \\
Z \ar[r] & X/U}
\end{equation}
is homotopy cocartesian for the motivic (injective) model structure on 
$\sM^G_k$.
In other words, the map ${X'}/U {\underset{p^{-1}(Z)}\coprod} Z \to X/U$
is a motivic weak equivalence.
\end{thm}

\vskip .4cm

\section{$eN$-descent and unstable representability for 
equivariant $K$-theory}\label{section:Nis-Desc-rep}
In this section, we establish Nisnevich descent for equivariant
$K$-theory for $k$-schemes (not necessarily smooth).
It follows that equivariant $K$-theory of smooth schemes is represented by an 
object in the equivariant motivic homotopy category.
As an application, we characterize all equivariantly contractible smooth affine 
curves with a group action, and moreover all equivariant vector bundles on such 
curves.

\subsection{The equivariant $K$-theory presheaf on $\Sm^G_k$}
\label{subsection:Construction-K}
Quillen's $Q$-construction associates to an exact category $\sE$ with a chosen 
zero object $\{0\}$,
the category $Q\sE$ whose objects are same as those of $\sE$ but
the morphisms between two objects $M'$ and $M''$ are diagrams
$M' \leftarrow N \rightarrow M''$, where the first arrow is an {\sl
admissible epimorphism} and the second arrow is an {\sl admissible 
monomorphism}. Taking the classifying space of $Q\sE$, one obtains a 
simplicial space $BQ\sE$ and the $K$-theory space of $\sE$ is defined
as 
\begin{equation}\label{eqn:QK-theory}
K(\sE) = \Omega BQ\sE.
\end{equation}
Alternate approaches include the $S_{\bullet}$-construction of Waldhausen 
\cite{Waldhausen} and the $G$-construction of Gillet-Grayson \cite{GG}. 
An advantage of the $G$-construction is that the resulting
simplicial space $G\sE$ defining $K$-theory is homotopy equivalent to 
$\Omega BQ\sE$. 

Let $G$ be a smooth affine group scheme over $k$.
If $\sE_X$ is the exact category of 
$G$-equivariant vector bundles on any $X \in \Sch^G_k$, there is a simplicial
set  $G\sE_X$ homotopy equivalent to $\Omega BQ\sE_X$.
With either of these approaches, 
algebraic $K$-theory is only a pseudo-presheaf of 
simplicial sets (or spectra) on the category of $G$-schemes, 
and not an honest simplicial presheaf.
This is remedied by rectification of pseudo-functors as in e.g., 
\cite{Thomason2}, \cite{Jardine4}. 
The equivariant $K$-groups of $X$ are given by 
$K^G_i(X) = \pi_i\left(G\sE_X\right)$ for $i \ge 0$.

We apply the rectification procedure to the pseudo-functor 
on $\Sch^G_k$ which takes a $G$-scheme $X$ to the exact symmetric 
monoidal category $\sP^G(X)$ of $G$-equivariant vector bundles. 
Using \cite[Lemma~3.2.6]{Thomason2} and
the rectification procedure as explained in 
\cite[Chapter~5, page ~179]{Jardine4}, this 
process yields the following equivariant $K$-theory presheaf on $\Sch^G_k$, 
and hence on $\Sm^G_k$ via the full embedding $\Sm^G_k \inj \Sch^G_k$.

\begin{prop}\label{prop:Equiv-rect}
There is a presheaf of pointed simplicial sets $\sK^G$ on $\Sch^G_k$
with a monoidal structure
\[
\sK^G \times \sK^G \stackrel{\wedge}{\to} \sK^G
\]
such that for every $X \in \Sch^G_k$ and $i \ge 0$, there is a canonical 
isomorphism $\pi_i(\sK^G(X)) \xrightarrow{\cong} K^G_i(X)$.
\end{prop}

\subsection{Nisnevich descent theorems for equivariant $K$-theory}
\label{subsection:END}
We refer to Theorem ~\ref{thm:LFMS-Psh} for the $eN$-local injective model 
structures on $\Sch^G_k$ and $\Sm^G_k$. 
The descent problem amounts to the following in our setting.
\begin{defn}\label{defn:Descent}
A simplicial presheaf $\sF$ on $\Sch^G_k$ or $\Sm^G_k$ satisfies $eN$-descent 
if every fibrant replacement $\sF \to {\rm Ex}{\sF}$ 
in the $eN$-local injective model structure is an objectwise weak equivalence.
\end{defn}

We shall say that a scheme $X \in \Sch^G_k$ is locally $G$-affine if every 
point in $X$ has a $G$-invariant affine neighborhood. 
Observe that all quasi-projective $G$-schemes are locally $G$-affine if
$G$ is finite. Moreover, 
all locally $G$-affine schemes admit good quotients for the $G$-action.
We denote the category of locally $G$-affine schemes by $\Sch^G_{k,\lAff}$.

\begin{prop}\label{prop:MVP}
Let $f : Y \to X$ be a $G$-equivariant \'etale morphism in $\Sch^G_k$ 
such that one of the following holds.
\begin{enumerate}
\item
$X$ and $Y$ are locally $G$-affine and $G$ is finite with
$({\rm char}(k), |G|) = 1$.
\item
$X$ and $Y$ are smooth.
\end{enumerate} 
Suppose there is a Cartesian square in $\Sm^G_k$
\[
\xymatrix@C.9pc{
W \ar[r] \ar[d] & Y \ar[d]^{f} \\
U \ar[r]_{j} & X} 
\]
where $j$ is an open immersion such that the map $f$ is an isomorphism
over the complement of $U$ (with the reduced structures).
Then the diagram of simplicial sets
\[
\xymatrix@C.9pc{
\sK^G(X) \ar[r] \ar[d] & \sK^G(Y) \ar[d] \\
\sK^G(U) \ar[r] & \sK^G(W)}
\]
is homotopy Cartesian.
\end{prop}
\begin{proof}
If we are in the case (1) of the theorem, then our assumption
implies that $[Y/G] \to [X/G]$ is a representable morphism of tame 
Deligne-Mumford stacks which admit coarse moduli schemes. 
Hence the theorem follows from \cite[Corollary~3.8]{AO}.

Suppose now that we are in case (2). We then have a commutative diagram
of fibration sequences (see \cite{Thomason3}):
\begin{equation}\label{eqn:K-descent-1}
\xymatrix@C1pc{
\sG^G(X \setminus U) \ar[r] \ar[d] & \sK^G(X) \ar[r] \ar[d] & \sK^G(U) \ar[d] \\
\sG^G(Y \setminus W) \ar[r] & \sK^G(Y) \ar[r] & \sK^G(W)}
\end{equation}
where $\sG^G(X)$ denotes the $K$-theory of the exact category of the equivariant 
coherent sheaves on a $G$-scheme $X$.
Our hypothesis implies that the left vertical arrow is a weak equivalence.
The theorem now follows.
\end{proof}

The following are the main results of this section.
\begin{thm}\label{thm:NIS-DESC}
Let $G$ be a smooth affine group scheme over $k$.
The simplicial presheaf $\sK^G$ satisfies $eN$-descent on $\Sm^G_k$. 

If $G$ is finite with $({\rm char}(k), |G|) = 1$, then $\sK^G$ satisfies
$eN$-descent on $\Sch^G_{k,\lAff}$.
\end{thm}
\begin{proof}
Let $\sK^{G} \to {\rm Ex}\sK^{G}$ be a fibrant replacement in the
local injective model structure on the simplicial presheaves on 
$\Sch^G_k$ (or $\Sm^G_k$). 
We have to show that this is a schemewise  weak equivalence.  
By Proposition~\ref{prop:Flasq-fib}, we only have to show that 
$\sK^{G}$ is flasque. We remark here that even though this
proposition is stated for smooth schemes, it is valid (with same proof)
for all $G$-schemes (see \cite[Lemma~3.5]{Voev2}).
Using Proposition~\ref{prop:Equiv-rect} and \cite[Proposition~13.3.13]{Hirsc},
it is enough to show that $\sK^G$ takes a distinguished $eN$-square 
~\eqref{eqn:cd-square} to a homotopy Cartesian square. 
This follows directly from Proposition~\ref{prop:MVP}.
\end{proof}

\begin{thm}\label{thm:M-Uns-Rep1}
Let $k$ be any field and $G$ a smooth affine group scheme over $k$. 
For any $X \in \Sm^G_k$ and $i \ge 0$, there is a 
canonical isomorphism
\[
K^G_i(X) \xrightarrow{\simeq} 
\Hom_{{\rm Ho}^G_{\A^1, \bullet}(k)} \left(S^i_s \wedge X_+, \sK^{G} \right).
\]
\end{thm}
\begin{proof}
Let $\sK^{G} \to {\rm Ex}\sK^{G}$ 
be a fibrant replacement of $\sK^{G}$ in the motivic injective model structure
on $\sM^G_{k, \bullet}$. The homotopy invariance property of equivariant
$K$-theory for smooth $G$-schemes implies that the motivic $G$-space
$\sK^{G}$ is $\A^1$-weak invariant (see Definition~\ref{defn:A-1model-str}).
Combining this with Proposition~\ref{prop:MVP}, we deduce that $\sK^G$
is $\A^1$-flasque (see \S~\ref{subsubsection:A1FL}).
We conclude from Theorem~\ref{thm:A1-flasq} that the map 
$\sK^G \to {\rm Ex}\sK^G$ is a schemewise weak equivalence.

We now apply Proposition~\ref{prop:DQA-Psh} to get a canonical isomorphism
\[
\begin{array}{lll}
\Hom_{{\rm Ho}^G_{\A^1, \bullet}(k)} \left(S^i_s \wedge X_+, 
\sK^{G} \right) & \xrightarrow{\simeq} & 
\Hom_{{\rm Ho}^G_{\A^1, \bullet}(k)} \left(S^i_s \wedge X_+, 
{\rm Ex}\sK^{G} \right) \\
& \xrightarrow{\simeq} &
[S^i_s, {\rm Ex}\sK^{G}(X)] \\
& \xrightarrow{\simeq} & \pi_i\left({\rm Ex}\sK^{G}(X)\right) \\
& \xleftarrow{\simeq} & \pi_i\left(\sK^{G}(X)\right) \\
& = & K^G_i(X).
\end{array}
\]
This completes the proof of the theorem.
\end{proof}

\subsection{Algebraic analogue of Segal's theorem}
\label{subsection:Segal}
Recall that if $G$ is a topological group, 
two $G$-equivariant continuous maps $\phi_0, \phi_1: X \to Y$
between topological $G$-spaces are called $G$-homotopic if
there exists a continuous $G$-equivariant map $H: X \times [0,1] \to Y$
(with trivial $G$-action on $[0,1]$) such that $H \circ i_j = \phi_j$
for $i_j:\{j\}\inj [0,1]$, $j=0,1$.

It was shown by Segal in \cite[Proposition~2.3]{Segal} that 
$G$-homotopic maps induce the same maps on equivariant topological 
$K$-theory. As an application of Theorem~\ref{thm:M-Uns-Rep1}, we prove the
following algebraic analogue of Segal's theorem.

\begin{cor}\label{cor:Segal-I}
Let $G$ be a smooth affine group scheme over a field $k$, 
and $\phi_0, \phi_1 :X \to Y$ equivariantly $\A^1$-homotopic maps in $\Sm^G_k$ (see \S~\ref{subsubsection:EBM}). 
Then $\phi^*_0 = \phi^*_1: K^G_*(Y) \to K^G_*(X)$.
\end{cor}
\begin{proof}
It is enough to consider the case when $\phi_0$ and $\phi_1$ are elementary
$\A^1$-homotopic. Let $i_0, i_1: \Spec(k) \to \A^1$ be the two inclusions
with $i_0(pt) = 0$ and $i_1(pt) = 1$.
It suffices to show that $i^*_0 = i^*_1: K^G_*(X \times \A^1) \to K^G_*(X)$.

Let $p: X \times \A^1 \to X$ denote the projection map. It follows from
Theorem~\ref{thm:M-Uns-Rep1} that $p^*: K^G_*(X) \to  K^G_*(X \times \A^1)$
is an isomorphism. Hence, it suffices to show that 
$(p\circ i_0)^* = (p \circ i_1)^*: K^G_*(X) \to K^G_*(X)$.
Both these maps equal the identity on $K^G_*(X)$.
\end{proof}

\subsection{Equivariantly contractible affine curves}\label{subsection:ECAC}
We shall say that a motivic $G$-space $\sX$ is 
{\sl equivariantly $\A^1$-contractible} if the map $\sX \to pt$ is a motivic 
weak equivalence. 
A $G$-equivariant vector bundle $\sV$ on $X\in\Sm^G_k$ is called trivial 
if there is 
a $G$-representation $V$ such that $\sV=V{\underset{k}\times}\ X$.

As an application of Theorem~\ref{thm:M-Uns-Rep1}, we prove the following 
desired geometric result on equivariant vector bundles.

\begin{thm}\label{thm:Curves}
Let $k$ be an infinite field and let $G = \<\sigma\>$ be a finite cyclic group 
of order prime to the characteristic of $k$ such that $\mu_{|G|} \subset k$.
Let $X$ be a smooth affine curve over $k$ with $G$-action. Then $X$ is
equivariantly $\A^1$-contractible if and only if it is isomorphic to
an 1-dimensional linear representation of $G$. In particular, all
$G$-equivariant vector bundles on $X$ are trivial if $X$ is equivariantly
$\A^1$-contractible. 
\end{thm}
\begin{proof}
The assertion that a finite-dimensional representation of $G$ is 
equivariantly $\A^1$-contractible, follows immediately from 
Proposition~\ref{prop:Hom-inv}. 
Below we prove the more difficult converse statement.

Suppose that $X$ is equivariantly $\A^1$-contractible.
Since the action of $G$ on a smooth scheme is linearizable, we can assume that 
there is smooth projective curve $\ov{X} \in \Sm^G_k$ and an open
immersion $j:X \inj \ov{X}$ in $\Sm^G_k$. Let $f: X \to \Spec(k)$ be the
structure map. 

{\bf {Claim~1:}}
The curve $X$ is rational.
\\
{\sl Proof of claim 1:}
Consider the commutative diagram
\begin{equation}\label{eqn:Curve-1}
\xymatrix@C1.3pc{
K^G_i(k) {\underset{R(G)}\otimes} \ \Z \ar[d]_{f^*} \ar[r]^>>>{f^i_k} & 
K_i(k) \ar[d]^{f^*} \\
K^G_i(X) {\underset{R(G)}\otimes} \ \Z \ar[r]_>>>{f^i_X} & K_i(X)} 
\end{equation}
with forgetful horizontal maps from equivariant to ordinary $K$-theory. 
Theorem~\ref{thm:M-Uns-Rep1} shows the left vertical arrow is an isomorphism for all $i \ge 0$.
The top horizontal arrow is an isomorphism for all $i \ge 0$ by \cite[Lemma~5.6]{Thomason4}. 
Applying these facts for $i=0$, 
we see that the composite map 
$$
K^G_0(X) {\underset{R(G)}\otimes} \ \Z \to K_0(X) \to \Z
$$ 
is an isomorphism. 
On the other hand, the first map is surjective over $\Z[{1}/{|G|}]$ by \cite[Theorem~1]{Vistoli1}.
It follows that $\Pic(X)$ is a torsion group of exponent $|G|$, 
which happens if and only if $X$ is rational. 
This proves the claim.

{\bf{Claim~2:}}
$X$ is isomorphic (not necessarily equivariantly) to $\A^1$.
\\
{\sl Proof of claim 2:}
Claim~1 implies that $\ov{X} \simeq \P^1_k$.
Inserting $i = 1$ in~\eqref{eqn:Curve-1} shows the composite map 
$$
K^G_1(X) {\underset{R(G)}\otimes} \ \Z \xrightarrow{f^1_X} K_1(X)\surj \sO^{\times}(X)
$$
is just the inclusion $k^{\times} \inj \sO^{\times}(X)$. 
On the other hand,
$f^1_X$ is surjective over $\Z[{1}/{|G|}]$ by \cite[Theorem~1]{Vistoli1}. 
It follows that $k^{\times}[{1}/{|G|}] \simeq\sO^{\times}(X)[{1}/{|G|}]$, 
which happens if and only if $X \simeq \A^1$ as an open subscheme of $\P^1_k$. 

By the above claims $X$ is the affine line with $G= \<\sigma\>$-action $\sigma(x)=ax+b$ for some fixed $a,b\in k$ with $a^{|G|} = 1$. 
If $b \neq 0$, then $\sigma$ acts on $\A^1$ without fixed points.
This means that the identity map of $\A^1$ gives an element of $[\A^1, \A^1]_{G, \A^1}$ which can not be equivariantly contracted 
to any fixed point.
In particular, 
$\pi^{G, \A^1}_0(X)$ is not constant and hence $X \to \Spec(k)$ is not a motivic weak equivalence, 
which contradicts our assumption. 
We conclude that $b=0$ and $G$ acts linearly on $\A^1$. 

Finally, the claim about the triviality of all $G$-equivariant vector bundles on $X$ follows from the above combined with \cite{Cast} 
and \cite[Theorem~1]{MMP}.
\end{proof}

\begin{exm}\label{exm:Non-contractible}
Theorem~\ref{thm:Curves} shows that
equivariant $\A^1$-contractibility is a strictly stronger condition than
ordinary $\A^1$-contractibility, as one would expect.
As an example, let the cyclic group of order two $G = \<\sigma\>$ act on 
$\A^1$ by $\sigma(x) = 1-x$. This action is fixed point free and hence
not isomorphic to a $G$-representation. Thus $\A^1$ with this action is
not equivariantly $\A^1$-contractible.
\end{exm}

\begin{remk}\label{remk:High-dim}
One can ask whether the assertion of Theorem~\ref{thm:Curves}
is true in higher dimensions as well. This seems to be a very difficult question.
We do not know the answer even when $G$ is trivial and $X$ is a surface.
That is, it is unknown whether an $\A^1$-contractible smooth affine surface is 
isomorphic to the affine plane. It is known, however, 
that such surfaces do not admit any non-trivial vector bundle.
\end{remk}

\vskip .4cm

\noindent\emph{Acknowledgments.}
The first version of this paper was written when AK was
visiting the Department of Mathematics at University of Oslo in the summer of
2011. He thanks the department for the invitation and support.
PA{\O} thanks David Gepner and Jeremiah Heller for discussions on 
subjects related to this paper. The authors
thank Aravind Asok and Ben Williams for inspiration and helpful comments
on an earlier version of this paper.

\begin{center}
School of Mathematics, Tata Institute of Fundamental Research, \\ 
Homi Bhabha Road, Colaba, Mumbai, India.\\
e-mail: amal@math.tifr.res.in
\end{center}
\begin{center}
Department of Mathematics, University of Oslo, Norway.\\
e-mail: paularne@math.uio.no
\end{center}

\end{document}